\documentclass[10pt]{article}
\usepackage{latexsym,amssymb,amsmath}
\usepackage{amssymb,amsthm,upref,amscd}
\everymath{\displaystyle}
\usepackage{color}
\usepackage[T1]{fontenc}
\begin{document}
\baselineskip=13pt

\numberwithin{equation}{section}

\pretolerance1000
\newtheorem{theorem}{Theorem}[section]
\newtheorem{lemma}[theorem]{Lemma}
\newtheorem{proposition}[theorem]{Proposition}
\newtheorem{corollary}[theorem]{Corollary}
\newtheorem{remark}[theorem]{Remark}
\newtheorem{definition}[theorem]{Definition}
\newtheorem{claim}[theorem]{Claim}
\renewcommand{\theequation}{\thesection.\arabic{equation}}

\newcommand{\A}{\mathbb{A}}
\newcommand{\B}{\mathbb{B}}
\newcommand{\C}{\mathbb{C}}
\newcommand{\D}{\mathbb{D}}
\newcommand{\E}{\mathbb{E}}
\newcommand{\F}{\mathbb{F}}
\newcommand{\G}{\mathbb{G}}
\newcommand{\I}{\mathbb{I}}
\newcommand{\J}{\mathbb{J}}
\newcommand{\K}{\mathbb{K}}
\newcommand{\M}{\mathbb{M}}
\newcommand{\N}{\mathbb{N}}
\newcommand{\Q}{\mathbb{Q}}
\newcommand{\R}{\mathbb{R}}
\newcommand{\T}{\mathbb{T}}
\newcommand{\U}{\mathbb{U}}
\newcommand{\V}{\mathbb{V}}
\newcommand{\W}{\mathbb{W}}
\newcommand{\X}{\mathbb{X}}
\newcommand{\Y}{\mathbb{Y}}
\newcommand{\Z}{\mathbb{Z}}
\newcommand\ca{\mathcal{A}}
\newcommand\cb{\mathcal{B}}
\newcommand\cc{\mathcal{C}}
\newcommand\cd{\mathcal{D}}
\newcommand\ce{\mathcal{E}}
\newcommand\cf{\mathcal{F}}
\newcommand\cg{\mathcal{G}}
\newcommand\ch{\mathcal{H}}
\newcommand\ci{\mathcal{I}}
\newcommand\cj{\mathcal{J}}
\newcommand\ck{\mathcal{K}}
\newcommand\cl{\mathcal{L}}
\newcommand\cm{\mathcal{M}}
\newcommand\cn{\mathcal{N}}
\newcommand\co{\mathcal{O}}
\newcommand\cp{\mathcal{P}}
\newcommand\cq{\mathcal{Q}}
\newcommand\rr{\mathcal{R}}
\newcommand\cs{\mathcal{S}}
\newcommand\ct{\mathcal{T}}
\newcommand\cu{\mathcal{U}}
\newcommand\cv{\mathcal{V}}
\newcommand\cw{\mathcal{W}}
\newcommand\cx{\mathcal{X}}
\newcommand\ocd{\overline{\cd}}

\def\c{\centerline}
\def\ov{\overline}
\def\emp {\emptyset}
\def\pa {\partial}
\def\bl{\setminus}
\def\op{\oplus}
\def\sbt{\subset}
\def\un{\underline}
\def\al {\alpha}
\def\bt {\beta}
\def\de {\delta}
\def\Ga {\Gamma}
\def\ga {\gamma}
\def\lm {\lambda}
\def\Lam {\Lambda}
\def\om {\omega}
\def\Om {\Omega}
\def\sa {\sigma}
\def\vr {\varepsilon}
\def\va {\varphi}

\title{\bf Multi-bump solutions for Choquard equation with deepening potential well}

\author { Claudianor O. Alves$^a$\thanks{C. O. Alves was partially supported by CNPq/Brazil
 301807/2013-2 and INCT-MAT, coalves@mat.ufcg.edu.br}\, , \, \ Al\^annio B. N\'obrega$^a$\thanks{alannio@mat.ufcg.edu.br}\, , \,    Minbo Yang$^b$\thanks{M. Yang is the corresponding author, he was  partially supported by NSFC (11571317, 11271331) and ZJNSF(LY15A010010), mbyang@zjnu.edu.cn}\vspace{2mm}
\and {\small $a.$ Universidade Federal de Campina Grande} \\ {\small Unidade Acad\^emica de Matem\'{a}tica} \\ {\small CEP: 58429-900, Campina Grande - Pb, Brazil}\\
{\small $b.$ Department of Mathematics, Zhejiang Normal University} \\ {\small Jinhua, 321004, P. R. China.}}

\date{}
\maketitle

\begin{abstract}
In this paper we study the existence of multi-bump solutions for the following Choquard equation
$$
\begin{array}{ll}
-\Delta u + (\lambda a(x)+1)u=\displaystyle\big(\frac{1}{|x|^{\mu}}\ast |u|^p\big)|u|^{p-2}u  \mbox{ in } \,\,\, \R^3,
\end{array}
$$
where $\mu \in (0,3), p\in(2, 6-\mu)$, $\lambda$ is a positive parameter and the nonnegative continuous function $a(x)$ has a potential well $ \Omega:=int (a^{-1}(0))$  which possesses $k$ disjoint bounded components $ \Omega:=\cup_{j=1}^{k}\Omega_j$. We prove that if the parameter $\lambda$ is large enough, then the equation has at least $2^{k}-1$ multi-bump solutions.

 \vspace{0.3cm}


\vspace{0.3cm}

\end{abstract}

\section{Introduction}
The nonlinear Choquard equation
\begin{equation}\label{Nonlocal.S1}
 -\Delta u +V(x)u =\Big(\frac{1}{|x|^{\mu}}\ast|u|^{p}\Big)|u|^{p-2}u  \quad \mbox{in} \quad \R^3,
\end{equation}
$p=2$ and $\mu=1$, goes
back to the description of the quantum theory of a polaron at rest by S. Pekar in 1954 \cite{P1}
and the modeling of an electron trapped
in its own hole in 1976 in the work of P. Choquard, as a certain approximation to Hartree-Fock theory of one-component
plasma \cite{L1}. In some particular cases, this equation is also known as the Schr\"{o}dinger-Newton equation, which was introduced by Penrose in his discussion on the selfgravitational collapse of a quantum mechanical wave function  \cite{Pe}.

The existence and qualitative properties of solutions of \eqref{Nonlocal.S1} have been widely studied in the last decades. In \cite{L1}, Lieb proved the existence and uniqueness, up to translations,
of the ground state. Later, in \cite{Ls}, Lions showed the existence of
a sequence of radially symmetric solutions. In \cite{CCS1, ML,  MS1} the authors showed the regularity, positivity
and radial symmetry of the ground states and
derived decay property at infinity as well. Moreover, Moroz and Van
Schaftingen in \cite{MS2} considered  the existence of ground states under the assumptions of Berestycki-Lions type. When $V$ is a continuous periodic function with $\inf_{\R^3} V(x)> 0$, noticing that the nonlocal term is invariant under translation, we can obtain easily the existence result  by applying the Mountain Pass Theorem, see \cite{AC} for example. For periodic potential $V$ that changes sign and $0$ lies in the gap of the spectrum of the Schr\"{o}dinger operator $-\Delta +V$, the problem is strongly indefinite, and the existence of solution for $p=2$ was considered in \cite{BJS} by reduction arguments. For a general case, Ackermann \cite{AC} proposed a new approach to prove the existence of infinitely many geometrically distinct weak solutions. For other related results, we refer the readers to \cite{CS, GS} for the existence of sign-changing solutions, \cite{MS3, S, WW} for the existence and concentration behavior of the semiclassical solutions and \cite{MS4} for the critical nonlocal part with respect to the Hardy-Littlewood-Sobolev inequality.

In the present paper, we are interested in the nonlinear Choquard equation with deepening potential well
$$
\begin{array}{ll}
  - \Delta u + (\lambda a(x)+1)u =\displaystyle\Big(\frac{1}{|x|^{\mu}}\ast|u|^{p}\Big)|u|^{p-2}u  \,\,\,   \mbox{ in } \,\,\, \R^3, \\
\end{array}\eqno{(C)_{\lambda}}
$$
where  $\mu \in (0,3), p \in(2, 6-\mu)$ and $a(x)$ is a  nonnegative continuous function with $\Omega = int (a^{-1}(0))$ being a non-empty bounded open set with smooth boundary $\partial \Omega$. Moreover, $\Omega$ has $k$ connected components, more precisely,
\begin{equation} \label{a1}
\Omega =\bigcup_{j=1}^{k}\Omega_j
\end{equation}
with
\begin{equation} \label{a2}
dist( \Omega_i, \Omega_j)>0  \quad \mbox{for} \quad i \neq j.
\end{equation}
Moreover, we suppose that there exists $M_0>0$ such that
\begin{equation}\label{a3}
	|\{ x\in \mathbb{R}^3;\, a(x)\leq
	M_0\}|<+\infty.
\end{equation}
Hereafter, if $A \subset \mathbb{R}^{3}$ is a mensurable set, $|A|$ denotes its Lebesgue's measure. The purpose of the present paper is to study the existence and the asymptotic shape of the solutions for $(C)_\lambda$ when $\lambda$ is large enough, more precisely, we will show the existence of multi-bump type solutions.

The motivation of the present paper arises from the results for the local Schr\"odinger equations with deepening potential well
\begin{equation} \label{LL2}
  - \Delta u + (\lambda a(x)+b(x))u = |u|^{p-2}u \,\,\,  \mbox{ in } \,\,\, \R^N,
\end{equation}
where $a(x), b(x)$ are suitable continuous functions and $ p \in (2, \frac{2N}{N-2} ) $ if $ N \ge 3 $; $ p \in (1, \infty) $ if $ N = 1, 2 $.
In \cite{BW2}, for $b(x)=1$, Bartsch and Wang proved the existence of a least energy solution for large $\lambda$ and that the sequence of solutions converges strongly to a least energy solution for a problem in bounded domain. They also showed the existence of at least cat$\Omega$ positive solutions for large $\lambda$, where $ \Omega =int (a^{-1}(0))$,  and the exponent $p$ is close to the critical exponent. The same results were also established by Clapp and Ding \cite{CD} for critical growth case. We also refer to \cite{BPW} for nonconstant $b(x)>0$, where the authors prove the existence of $k$ solutions that may change sign for any $k$ and $\lambda$ large enough. For other results related to Schr\"odinger equations with deepening potential well, we may refer the readers to \cite{ST, SZ, WZ}

 The existence and characterization of the solutions for problem \eqref{LL2} with large parameter $\lambda$ were considered in \cite{Alves, DT},  by supposing that $a(x)$ has a potential well $ \Omega =int (a^{-1}(0))$ consisting of $k$ disjoint bounded components $\Omega_1,\cdots, \Omega_k$, the authors studied the multiplicity and multi-bump shape of the solutions associated to the number of the components of the domain $ \Omega =int (a^{-1}(0))$. In \cite{DT}, by using of penalization ideas developed in \cite{DF1}, Ding and Tanaka were able to overcome the loss of compactness and then they applied  the deformation flow arguments found in \cite{ZR, Se} to prove the existence at least $ 2^k-1 $ solutions $u_\lambda$ for large values of $ \lambda $.  More precisely, for each non-empty subset $ \Gamma $ of $ \{ 1,\ldots,k \} $, it was proved that, for any sequence $ \lambda_n \to \infty $ one can extract a subsequence $( \lambda_{n_i}) $ such that $( u_{ \lambda_{n_i} } )$ converges strongly in $ H^1 \big( \mathbb R^N \big) $ to a function $ u $, which satisfies $ u = 0 $ outside $\displaystyle \Omega_\Gamma= \bigcup_{ j \in \Gamma } \Omega_j $ and $ u_{|_{\Omega_j}}, \, j \in \Gamma$, is a least energy solution for
\begin{equation} \label{LL}
   \begin{cases}
		  - \Delta u +  u  = |u|^{p-2}u, \text{ in } \Omega_j,\vspace{3mm} \\
		  u \in H^1_0 \big( \Omega_j \big), \, u > 0, \text{ in } \Omega_j.
	 \end{cases}
\end{equation}
As we all know, the problem \eqref{LL}  on bounded domain plays an important role in the study of multi-bump shaped solutions for problem \eqref{LL2}. By using of "gluing" techniques, Ding and Tanaka  used the ground states of problem \eqref{LL} as building bricks to construct minimax values and then proved the existence of multibump solutions by deformation flow arguments.

From the commentaries above, it is quite natural to ask if the results in \cite{Alves, DT} still hold for the generalized Choquard equation. Unfortunately, we can not draw a similar conclusion in a straight way, since  the nonlinearity of the generalized Choquard equation is a  nonlocal one. For $\Gamma=\{1,\cdots,l\}$ with $l \le k$ and $\displaystyle \Omega_\Gamma= \bigcup_{ j \in \Gamma } \Omega_j $ , it is easy to see that
$$
\int_{\Omega_\Gamma}\Big(\int_{\Omega_\Gamma} \frac{ |u|^p}{|x-y|^{\mu}}dy\Big)|u|^{p}dx\neq\displaystyle\sum_{i=1}^l\int_{\Omega_i}\Big(\int_{\Omega_i} \frac{ |u|^p}{|x-y|^{\mu}}dy\Big)|u|^{p}dx.
$$
Thus, it is impossible to repeat the same arguments explored in \cite{DT} to use the least energy solution of
$$
   \begin{cases}
		  - \Delta u +  u  = \displaystyle\Big(\int_{\Omega_j} \frac{ |u|^p}{|x-y|^{\mu}}dy\Big)|u|^{p-2}u, \text{ in } \Omega_j,\\
			\, u \not= 0, \text{ in } \Omega_j, \\
		  u \in H^1_0 \big( \Omega_j \big), \quad j\in \Gamma,
	 \end{cases} \eqno{(C)_{j}}
$$
for building  the multi-bump solutions. For the generalized  Choquard equation, it can be observed that the equation
  $$
   \begin{cases}
		  - \Delta u +  u  = \displaystyle\Big(\int_{\Omega_\Gamma} \frac{ |u|^p}{|x-y|^{\mu}}dy\Big)|u|^{p-2}u, \text{ in } \Omega_\Gamma,\\
		  u \in H^1_0 \big( \Omega_\Gamma \big),
	 \end{cases} \eqno{(C)_{\infty,\Gamma}}
$$
plays the role of the limit problem for equation $(C)_{\lambda}$ as $\lambda$ goes to infinity. Moreover, noticing that the solution may disappear on some component, for the Dirichlet problem of Choquard equation with components, it is not easy to prove the existence of  the least energy solution that is nonzero on each component $\Omega_j$, $j\in \Gamma$. In order to find this type of least energy solution  we will study the minimizing problem on a subset of the Nehari manifold, see Section 2 for more details.

Here, we will also avoid the penalization arguments found in \cite{DF1}, because by using this method we are led to assume more restrictions on the constants $\mu$ and $p.$ For that reason, instead of the penalization method, we  will follow the approach explored by Alves and N\'obrega in \cite{AN}, which showed the existence of  multi-bump solution for  problem  (\ref{LL2}) driven by the biharmonic operator.  Thus, as in \cite{AN},  we will work directly with the energy functional associated with $(C)_\lambda$, and we will modify in a different way the set of pathes where Deformation Lemma is used, see Sections 5 and 6 for more details.

To prove the existence of positive multi-bump solutions for $(C)_{\lambda}$, the first step is to consider the limit Dirichlet problem $(C)_{\infty,\Gamma}$  and to look for the existence of least energy solution  that is nonzero on each component $\Omega_j$, $j\in \Gamma$ . Having this in mind, we proved the following result.

\begin{theorem} \label{main1}
Suppose that $\mu \in (0,3)$ and  $p \in [2,6-\mu)$. Then problem $(C)_{\infty,\Gamma}$ possesses a least energy solution $u$ that is nonzero on each component $\Omega_{j}$ of $\Omega_{\Gamma}$, $j\in \Gamma$.
\end{theorem}


Using the above theorem, we are able to state our main result.

\begin{theorem} \label{main2}
Suppose that $\mu \in (0,3)$ and  $p \in (2,6-\mu)$. There exists a constant $ \lambda_0 > 0 $, such that for any non-empty subset $\Gamma \subset \lbrace1,\cdots, k\rbrace$  and $ \lambda \ge \lambda_0 $, the problem $ \big( C \big)_\lambda $ has a positive solution $u_{\lambda}$, which possesses the following property: For any sequence $ \lambda_n \to \infty $ we can extract a subsequence $ (\lambda_{n_i}) $ such that $ (u_{ \lambda_{n_i}}) $ converges strongly in $ H^{1}(\mathbb R^3) $ to a function $ u $, which satisfies $ u = 0 $ outside $\displaystyle \Omega_{\Gamma} =\bigcup_{j\in \Gamma}\Omega_j $, and $ u_{|_{ \Omega_{\Gamma}}}$ is a least energy solution for $(C)_{\infty,\Gamma}$ in the sense of Theorem \ref{main1}.
\end{theorem}
\begin{remark}\label{Re1}
Noting that $ (u_{ \lambda_{n_i}}) $ converges strongly in $ H^{1}(\mathbb R^3) $ to a function $ u $,
which is zero outside $\displaystyle \Omega_{\Gamma}$ and nonzero on each component $\Omega_j$, $j\in \Gamma$. In this way,  we can conclude that the solutions $ (u_{ \lambda}) $ have the shape of multi-bump  if $\lambda$ is large enough.
\end{remark}

\begin{remark}\label{Re2}
By the Hardy-Littlewood-Sobolev inequality, the natural interval for considering the Choquard equation is $ (\frac{6-\mu}{3},6-\mu)$, however, the case $\frac{6-\mu}{3}<p\leq2$ is not considered in Theorem \ref{main2}. This is due to the fact that  the method applied here do have some limitations in proving the intersection property for the pathes and the set $\mathcal{M}_\Gamma$ defined in Section 5. Inspired by a recent paper by Ghimenti,  Moroz and  Van Schaftingen \cite{GMS}, we will consider the case $p=2$ in a future paper by approximation with $p\downarrow2$.
\end{remark}


In order to apply variational methods to obtain the solutions for problems $(C)_{\lambda}$ and $(C)_{\infty,\Gamma}$ , the following classical Hardy-Littlewood-Sobolev inequality will be frequently used.
\begin{proposition}\label{HLS} \cite{LL} $\,\,[Hardy-Littlewood-Sobolev \ inequality]$:\\
Let $s, r>1$ and $0<\mu<3$ with $1/s+\mu/3+1/r=2$. Let $f\in
L^s(\R^3)$ and $h\in L^r(\R^3)$. There exists a sharp constant
$C(s,\mu,r)$, independent of $f,h$, such that
$$
\int_{\R^3}\int_{\R^3}\frac{f(x)h(y)}{|x-y|^\mu}dydx\leq
C(s,\mu,r) |f|_s|h|_r.
$$
  \end{proposition}

In the sequel,  we fix $ E_\lambda = \big( E, \| \cdot \|_\lambda \big) $ where
$$
   E = \left\{ u \in H^{1} ( \mathbb R^3 ) \, ; \, \int_{ \mathbb R^3 } a(x) |u|^{ 2 }dx < \infty \right\},
$$
and
$$
   \| u \|_\lambda = \left(\int_{\R^3}(|\nabla u|^{2}+(\lambda a(x)+1)|u|^{2})dx \right)^{\frac{1}{2}}.
$$
Obviously, $ E_\lambda $ is a Hilbert space, $ E_\lambda \hookrightarrow H^{1}( \mathbb R^3 ) $ continuously for $ \lambda \geq 0 $ and  $ E_\lambda $ is compactly embedded in $ L_{ loc }^{s}( \mathbb R^3 ) $, for all $ 1 \leq s <6$. We will study the existence of solutions for problem $(C)_\lambda$ by looking for critical points of the energy functional $ I_\lambda \colon E_\lambda \to \mathbb R $ given by
$$
   I_\lambda (u) = \frac{1}{2}\int_{ \mathbb R^3 }  \left( | \nabla u |^{2 } + \big( \lambda a(x) + 1 \big) u^{ 2 } \right)dx -\frac {1}{2p}\int_{\R^3}\Big(\frac{1}{|x|^{\mu}}\ast|u|^{p}\Big)|u|^{p}dx.
$$
For $p \in (\frac{6-\mu}{3},6-\mu)$, the Hardy-Littlewood-Sobolev inequality and the Sobolev embeddings imply that the functional  $I_\lambda \in C^1(E_\lambda,\mathbb{R})$ with
\[
I_\lambda'(u)v =  \int_{\R^3}\nabla u \nabla v +(\lambda a(x)+1)uv dx- \int_{\R^3}\Big(\frac{1}{|x|^{\mu}}\ast|u|^{p}\Big) |u|^{p-2}uvdx , \,\,\,\, \forall u,v \in E_\lambda.
\]
Hence, the critical points of $I_\lambda$ are in fact the weak solutions for problem $(C)_\lambda$.

This paper is organized as follows. In Section 2, we study a nonlocal problem set on bounded domain with 2 disjoint components for simplicity. By minimizing and deformation flow arguments, we are able to prove the existence of least energy solution which is nonzero on each component. In Section 3, we adapt the method used in \cite{AN} for the nonlocal situation, which permits us to prove that the energy functional satisfies the $(PS)$ condition for $\lambda$ large enough. In Section 4, we study the behavior of $(PS)_{\infty}$ sequence. In Section 5 and 6, we adapt the deformation flow method to establish the existence of a special critical point, which is crucial for showing the existence of multi-bump solutions  for $\lambda$ large enough.

\section { The problem $(C)_{\infty,\Gamma}$}
First, we need to study the Dirichlet  problem $(C)_{\infty,\Gamma}$ with several components and investigate the existence of least energy solution that is nonzero on each component. The main idea is to prove that the energy functional associated to $(C)_{\infty,\Gamma}$ defined by
\begin{equation} \label{LFJ}
I_\Gamma(u)=\frac 12 \int_{\Omega_\Gamma}(|\nabla u|^{2}+|u|^{2})dx -\frac {1}{2p} \int_{\Omega_\Gamma}\Big(\int_{\Omega_\Gamma} \frac{ |u|^p}{|x-y|^{\mu}}dy\Big)|u|^{p}dx
\end{equation}
achieves a minimum value on
$$
{\mathcal M_{\Gamma}}=\{u\in \mathcal N_{\Gamma}: I_\Gamma'(u)u_i=0 \mbox { and }
u_{i}\neq 0, \,\,\, i\in \Gamma\}
$$
where $\Gamma \subset \{1,\cdots,k\}$, $u_{i}=u{|_{ \Omega_i}}$ and $\mathcal{N}_{\Gamma}$ is the Nehari manifold of $I_\Gamma$ defined by
$$
{\mathcal N_{\Gamma}} = \{ u\in
H^1_0(\Omega_{\Gamma})\setminus\{0\}\, :\, I_\Gamma'(u)u=0\}.
$$
More precisely, we will prove that there is $w \in {\mathcal{M}_{\Gamma}}$ such that
\begin{equation} \label{LESC}
I_\Gamma(w)=\inf_{u \in {\mathcal{M}_{\Gamma}}}I_\Gamma(u) \quad \mbox{and} \quad I'_\Gamma(w)=0.
\end{equation}
Hereafter, we say that $w \in H^1_0(\Omega_{\Gamma})$, satisfying  $w_{i}=w{|_{ \Omega_i}}\neq 0, i\in \Gamma$, is a least energy solution for $(C)_{\infty,\Gamma}$ if the above condition \eqref{LESC}  holds. This feature will be used to characterize the multi-bump shape of the solutions of $(C)_\lambda$. Without loss of generality,  we will only consider $\Gamma=\{1,2\}$ for simplicity.  Moreover, we denote
by $\Omega,$ $\mathcal{M}$ and $\mathcal{N}$ the sets $\Omega_{\Gamma},$  $\mathcal{M}_{\Gamma}$ and $\mathcal{N}_{\Gamma}$ respectively, and $I_\Gamma$ will be denoted by $I$. Thereby,
$$
\Omega =\Omega_1\cup \Omega_2,$$
$$
{\mathcal M}=\{u\in \mathcal N: I'(u)u_i=0 \mbox { and }
u_{i}\neq 0, \,\,\, i=1,2. \}
$$
and
$$
{\mathcal N} = \{ u\in
H^1_0(\Omega)\setminus\{0\}\, :\, I'(u)u=0\}.
$$
In what follows, we denote by $||\,\,\,||$, $||\,\,\,||_1$ and $||\,\,\,||_2$  the norms in $H^{1}_0(\Omega)$, $H^{1}_0(\Omega_1)$ and $H^{1}_0(\Omega_2)$ given by
$$
||u||=\left(\int_{\Omega}(|\nabla u|^{2}+|u|^{2})dx\right)^{\frac{1}{2}},
$$
$$
||u||_1=\left(\int_{\Omega_1}(|\nabla u|^{2}+|u|^{2})dx\right)^{\frac{1}{2}}
$$
and
$$
||u||_2=\left(\int_{\Omega_2}(|\nabla u|^{2}+|u|^{2})dx\right)^{\frac{1}{2}}
$$
respectively.

The following Lemma shows that the set ${\mathcal M}$ is not empty.
\begin{lemma}\label{lema3}
Let $0<\mu<3$, $2 \leq p<6-\mu$ and  $v\in H_0^1(\Omega)$ with $v_{j}\neq 0$ for $j=1,2$, then  there exists $(\beta_1, \beta_2) \in (0,+\infty)^2$ such that
$\beta_1 v_1 + \beta_2 v_2 \in {\mathcal M}$ which means ${\mathcal M} \not= \emptyset$ and moreover,  $c_0=\inf_{u \in {\mathcal M}}I(u)>0$.
\end{lemma}
\begin{proof} Fix $v \in H^{1}_0(\Omega)$ with $v_i \not=0$ for $i=1,2$, and for the case $p=2$, without loss of generality, we may additionally assume that
\begin{equation} \label{condition}
\|v_i\|_i^{2}\not= \int_{\Omega_i}\Big(\int_{\Omega_i} \frac{ |v_i|^2}{|x-y|^{\mu}}dy\Big)|v_i|^{2}dx \quad \mbox{for} \quad i=1,2.
\end{equation}
Adapt  some ideas in \cite{GS} and \cite{GMS} by changing variables $t_j=s_j^{\frac{1}{p}}$, we define the function
$$
G(s_1,s_2)=I(s_1^{\frac{1}{p}}v_1+s_2^{\frac{1}{p}}v_2)=\frac{s_1^{\frac{2}{p}}}{2}\|v_1\|_1^{2}+\frac{s_2^{\frac{2}{p}}}{2}\|v_2\|_2^{2}-\frac{1}{2p}\int_{\Omega}\left(\frac{1}{|x|^{\frac{\mu}{2}}}*(s_1|v_1|^{p}+s_2|v_2|^{p})\right)^{2} dx.
$$
As $G$ is a continuous function and $G(s_1,s_2) \to 0$ as $|(s_1,s_2)| \to +\infty$, we have that $G$ possesses a global maximum point $(a,b) \in [0,+\infty)^{2}$. However, as $G$ is strictly concave function, it follows that  $(a,b) \in (0,+\infty)^{2}$, $(a,b)$ is the unique global maximum point and $\nabla G(a,b)=(0,0)$, which implies that  ${\mathcal M} \not= \emptyset$. Here, we would like to point out that if $p>2$, it is easy to check that $a,b \not =0$. While for the case $p=2$, we are able to show this fact only with the restriction (\ref{condition}). In fact, argue by contradiction that $a=0$, notice that $(0,b)$ is the maximum point of $G$ then there holds
$$
\|v_2\|_2^{2}=b^2\int_{\Omega_2}\Big(\int_{\Omega_2} \frac{ |v_2|^2}{|x-y|^{\mu}}dy\Big)|v_2|^{2}dx,
$$
therefore $b\neq 1$. Consider the function $g:[0,+\infty) \to \mathbb{R}$ given by
$$
g(t)=G(0,b+\alpha t),
$$
where $\alpha$ is to be determined later. A direct computation shows that
$$
g'(0)=\frac12\|v_1\|^2_1-\frac{b}{2}\int_{\Omega_1}\int_{\Omega_2}\frac{|v_2|^2(y)|v_1|^{2}(x)}{|x-y|^{\mu}}dydx+\frac{\alpha}{2}\Big(\|v_2\|_2^{2}- b\int_{\Omega_2}\Big(\int_{\Omega_2} \frac{ |v_2|^2}{|x-y|^{\mu}}dy\Big)|v_2|^{2}dx\Big).
$$
Consequently, if $\alpha$ is suitably chosen, we have
$$
g'(0)>0,
$$
which obviously is a contradiction.

Next, we will show that $c_0 >0$. To begin with, we recall that if $w \in {\mathcal M}$, then
$$
\|w\|^{2}=\int_{\Omega}\int_{\Omega}\frac{|w(y)|^p|w(x)|^p}{|x-y|^{\mu}}dydx.
$$
Using the Hardy-Littlewood-Sobolev inequality, there is $C>0$ such that
$$
\|w\|^{2} \leq C \|w\|^{4}.
$$
As $\|w\|\not= 0$, the last inequality yields there is $\tau >0$ satisfying
$$
\|w\|^{2} \geq \tau, \quad \forall w \in {\mathcal M}.
$$
From this,
$$
I(w)=I(w)-\frac{1}{2p}I'(w)w=\left(\frac{1}{2}-\frac{1}{2p} \right)\|w\|^{2}\geq \left(\frac{1}{2}-\frac{1}{2p} \right)\tau>0, \quad \forall w \in {\mathcal M}.
$$
and so,  $c_0 \geq \left(\frac{1}{2}-\frac{1}{2p} \right)\tau>0$.
\end{proof}

Let us state more a technical lemma.

 \begin{lemma}\label{lema2x}
Let $0<\mu<3$, $2 \leq p<6-\mu$ and $(w_n)$ be a bounded sequence in  ${\mathcal M}$ with $w_n \rightharpoonup w$ in $H^{1}_0(\Omega)$. If $\|w_{n,j}\| \not\to 0$, then $w_{j} \not= 0$ , where $w_{n,j}=w_n|_{\Omega_j}$ and $w_{j}=w|_{\Omega_j}$ for $j=1,2$.
\end{lemma}

\begin{proof} Assume by contradiction that $w_1=0$.  By the Hardy-Littlewood-Sobolev inequality and the Sobolev embeddings, we see that
$$
\int_{\Omega_j}\left(\int_{\Omega}\frac{|w_n|^p}{|x-y|^{\mu}}dy\right)|w_{n,1}|^pdx \to 0.
$$
On the other hand, as $I'(w_n)(w_{n,j})=0$, or equivalently
$$
\|w_{n,1}\|_1^2=\int_{\Omega_1}\left(\int_{\Omega}\frac{|w_n|^p}{|x-y|^{\mu}}dy\right)|w_{n,1}|^pdx,
$$
we derive that
$$
\|w_{n,1}\|_1^2 \to 0
$$
which is an absurd. The case $w_2 \not= 0$ is made of similar way.
\end{proof}

Now, we are able to show the existence of least energy solution for $(C)_{\infty,\Gamma}$.

\subsection{ Proof of Theorem \ref{main1}}

From Lemma \ref{lema3}, $c_0>0$ and there is a sequence $(w_n) \subset \mathcal M$ such that
$$
\lim_{n} I(w_n)=c_0.
$$
It is easy to see that $(w_n)$ is a bounded
sequence. Hence, without loss of generality, we may suppose that $w_n \rightharpoonup w \,\,\, \mbox{in} \,\,\, H_0^{1}(\Omega)$ and $
w_n \to w \,\,\, \mbox{in} \,\,\,  L^q(\Omega) \,\,\ \forall \, q \in [1,6)$, as $n\to \infty$. By considering the function
$$
G(s_1,s_2)=I(s_1^{\frac{1}{p}}(w_n)_1+s_2^{\frac{1}{p}}(w_n)_2)=\frac{s_1^{\frac{2}{p}}}{2}\|(w_n)_1\|_1^{2}+\frac{s_2^{\frac{2}{p}}}{2}\|(w_n)_2\|_2^{2}-\frac{1}{2p}\int_{\Omega}\left(\frac{1}{|x|^{\frac{\mu}{2}}}*(s_1|(w_n)_1|^{p}+s_2|(w_n)_2|^{p})\right)^{2} dx.
$$
we know by the previous study that $\nabla G(1,1)=(0,0)$. As $G$ is strictly concave function, $(1,1)$ is its global maximum point. Thus,
$$
I(w_n)=I((w_n)_1+(w_n)_2)=\max_{t,s \geq 0}I(t(w_n)_1+s(w_n)_2).
$$
Using the above information, we also know that $w_j \not= 0$ for $j=1,2$. Then, by Lemma \ref{lema3} there are $t_1,t_2>0$ such that
$$
t_1w_1+t_2w_2 \in \mathcal{M},
$$
and so,
$$
c_0 \leq I(t_1w_1+t_2w_2).
$$
By using the fact that $w_n \rightharpoonup w$ in $H^{1}_{0}(\Omega)$ and the compact Sobolev embeddings, we get
$$
I(t_1w_1+t_2w_2) \leq \liminf_{n \to +\infty}I(t_1(w_n)_1+t_2(w_n)_2) \leq \liminf_{n \to +\infty}I(w_n)=c_0,
$$
from where it follows that
$$
c_0 = I(t_1w_1+t_2w_2) \quad \mbox{with} \quad t_1w_1+t_2w_2 \in \mathcal{M}.
$$

Now, we we will show that $w_*=t_1w_1+t_2w_2$ is a critical point for $I$. Assume by contradiction  that $\|I'(w_*)\|>0$ and fix $\alpha>0$ such that
$$
\|I'(w_*)\| \geq \alpha.
$$
Moreover, we will fix $r>0$ small enough such that if $(t,s) \in B={B}_{r}(1,1) \subset \mathbb{R}^2$, then there exists some $\epsilon_0 >0$ such that
\begin{equation} \label{ETA0}
I(t^{\frac{1}{p}}(w_*)_1+s^{\frac{1}{p}}(w_*)_2)<c_0 - 2\epsilon_0, \quad \forall (t,s) \in \partial B.
\end{equation}
In the sequel we fix $\epsilon \in (0,\epsilon_0)$ and $\delta>0$ small emough such that
$$
\|I'(u)\| \geq \frac{\alpha}{2} \geq \frac{4\epsilon}{\delta} \quad \forall u \in  I^{-1}\{[c_0-2\epsilon,c_0+2\epsilon]\} \cap S
$$
where
$$
S=\{t^{\frac{1}{p}}(w_*)_1+s^{\frac{1}{p}}(w_*)_2\,:\, (t,s) \in \overline{B} \}.
$$
By using the Deformation Lemma, there exists a continuous map $\eta : H^{1}_{0}(\Omega) \to H^{1}_{0}(\Omega)$,  such that
\begin{equation} \label{ETA1}
\eta(u)=u \quad \forall u \notin I^{-1}\{[c_0-2\epsilon,c_0+2\epsilon]\} \cap S
\end{equation}
and
\begin{equation} \label{ETA2}
\eta(I^{c_0+\epsilon} \cap S) \subset I^{c_0-\epsilon} \cap S_\delta
\end{equation}
where
$$
S_\delta=\{v \in H_{0}^{1}(\Omega) \,:\,dist(v,S) \leq \delta \}.
$$
In the sequel, we fix $\delta >0$ of a way that
\begin{equation} \label{ETA3}
v \in S_\delta \Rightarrow v_1,v_2 \not= 0.
\end{equation}
Now, setting $\gamma(t,s)=\eta(t^{\frac{1}{p}}(w_*)_1+s^{\frac{1}{p}}(w_*)_2)$, (\ref{ETA0}) and (\ref{ETA1}) imply that
\begin{equation} \label{ETA4}
\gamma(s,t)= t^{\frac{1}{p}}(w_*)_1+s^{\frac{1}{p}}(w_*)_2, \quad \forall (s,t) \in \partial B.
\end{equation}
Moreover, since
$$
\max_{t,s\geq 0}I(t^{\frac{1}{p}}(w_*)_1+s^{\frac{1}{p}}(w_*)_2)=I(t_1()w_*)_1+t_2(w_*)_2)=c_0,
$$
by (\ref{ETA2}), we know
$$
I(\gamma(t,s)) \leq c_0-\epsilon.
$$
\begin{claim} \label{CL1}
There is $(t_0,s_0) \in B$ such that
$$
\big(I'(\gamma(t_0,s_0))(\gamma(t_0,s_0)_1),I'(\gamma(t_0,s_0))(\gamma(t_0,s_0)_2)\big)=(0,0).
$$
\end{claim}
Assuming for a moment the claim is true, we deduce that $\gamma(t_0,s_0) \in \mathcal{M}$, and so,
$$
c_0 \leq I(\gamma(t_0,s_0)) \leq c_0 - \epsilon
$$
which is absurd. Here,  (\ref{ETA3}) was  used to ensure that $\gamma(t_0,s_0)_j \not=0$ for $j=1,2$. Consequently, $w_*=t_1w_1+t_2w_2$ is a critical point for $I$.

\vspace{0.5 cm}

\noindent {\bf Proof of Claim \ref{CL1}:}\, First of all, note that
$$
\big(I'(\gamma(t,s))(\gamma(t,s)_1),I'(\gamma(t,s))(\gamma(t,s)_2)\big)=(0,0) \Leftrightarrow \big(\frac{1}{t}I'(\gamma(t,s))(\gamma(t,s)_1),\frac{1}{s}I'(\gamma(t,s))(\gamma(t,s)_2)\big)=(0,0)
$$
and by (\ref{ETA4})
$$
\big(\frac{1}{t}I'(\gamma(t,s))(\gamma(t,s)_1),\frac{1}{s}I'(\gamma(t,s))(\gamma(t,s)_2)\big)=\big(\frac{1}{t}I'(\gamma_0(t,s))(\gamma_0(t,s)_1),\frac{1}{s}I'(\gamma_0(t,s))(\gamma_0(t,s)_2)\big) \quad \forall (t,s) \in \partial B
$$
where
$$
\gamma_0(t,s)=t^{\frac{1}{p}}(w_*)_1+s^{\frac{1}{p}}(w_*)_2 \quad \forall (s,t) \in \overline{B}.
$$
Considering  the function
$$
G(t,s)=I(t^{\frac{1}{p}}(w_*)_1+s^{\frac{1}{p}}(w_*)_2),
$$
we deduce that
$$
(\frac{1}{t}I'(\gamma_0(t,s))(\gamma_0(t,s)_1),\frac{1}{s}I'(\gamma_0(t,s))(\gamma_0(t,s)_2))=\nabla G(t,s) \quad \forall (t,s) \in \overline{B}.
$$
Since $G$ is is strictly concave function and $\nabla G(1,1)=(0,0)$, it follows that
$$
0>\left\langle \nabla G(t,s)-  \nabla G(1,1),(t,s)-(1,1)\right\rangle = \left\langle \nabla G(t,s),(t,s)-(1,1)\right\rangle \quad \forall (s,t) \not= (1,1),
$$
and so,
$$
0>\left\langle \nabla G(t,s),(t,s)-(1,1)\right\rangle \quad \mbox{for} \quad |(s,t) -(1,1)|=r.
$$
Setting $H:\mathbb{R}^{2} \to \mathbb{R}^2$ by
$$
H(t,s)=\big(\frac{1}{t}I'(\gamma(t,s))(\gamma(t,s)_1),\frac{1}{s}I'(\gamma(t,s))(\gamma(t,s)_2)\big)
$$
and
$f(t,s)=H(t+1,s+1)$, we have that
$$
0>\left\langle f(t,s),(t,s)\right\rangle \quad \mbox{for} \quad |(s,t)|=r.
$$
By using the Brouwer's fixed point Theorem, we know there exists $(t_*,s_*) \in B_r(0,0)$ such that $f(t_*,s_*)=(0,0)$, that is,
$$
H(t_*+1,s_*+1)=(0,0),
$$
from where it follows that there is $(t_0,s_0) \in B$ such that
$$
\big(I'(\gamma(t,s))(\gamma(t,s)_1),I'(\gamma(t,s))(\gamma(t,s)_2)\big)=(0,0),
$$
which completes the proof of the claim.

\section{The $(PS)_c$ condition for $I_\lambda$}

In this section, we will prove some convergence properties for the  $(PS)$ sequences of the functional $I_\lambda$. Our main goal is to prove that, for given $c\geq 0$ independent of $\lambda$, the functional $I_{\lambda}$ satisfies the $(PS)_d$ condition for $d \in [0,c)$, provided that  $\lambda$ is large enough.
\begin{lemma}\label{al1}
	Let $(u_n) \subset E_{\lambda}$ be a $(PS)_c$ sequence for
	$I_{\lambda},$ then $(u_n)$ is bounded. Furthermore, $c\geq 0.$
\end{lemma}
\begin{proof}
Since $(u_n)$ is a $(PS)_c$ sequence,
	$$
	I_{\lambda}(u_n)\rightarrow c\ \mbox{and}\ I'_{\lambda}(u_n) \rightarrow 0.
	$$
	Then, for $n$ large enough
	\begin{equation}\label{3e1}
	I_{\lambda}(u_n)-\frac{1}{2p}I'_{\lambda}(u_n)u_n \leq c+1+\|u_n\|_{\lambda}.
	\end{equation}
	On the other hand,
	\begin{equation}\label{3e2}
	I_{\lambda}(u_n)-\frac{1}{2p}I'_{\lambda}(u_n)u_n=\left( \frac{1}{2}-\frac{1}{2p}\right)\|u_n\|^2_{\lambda}.
	\end{equation}
	Therefore, from $(\ref{3e1})$ and $(\ref{3e2})$ we get the inequality below
	$$
	\left( \frac{1}{2}-\frac{1}{2p}\right)\|u_n\|^2 _{\lambda} \leq c+1+\|u_n\|_{\lambda},
	$$
	which shows the boundedness of $(u_n)$. Thereby,  by (\ref{3e2}),
	\begin{equation}\label{3e3}
		0\leq\left(
	\frac{1}{2}-\frac{1}{2p}\right)\|u_n\|^2 _{\lambda} \leq
	c+o_n(1),
	\end{equation}
and the lemma follows by taking the limit of $n \to +\infty$.
\end{proof}
\begin{corollary}\label{ac1}
	Let $(u_n) \subset E_{\lambda}$ be a $(PS)_0$ sequence for
	$I_{\lambda}.$ Then $u_n\rightarrow 0$ in $E_\lambda$.
\end{corollary}
\begin{proof}
	An immediate consequence of the arguments used in the proof of Lemma $\ref{al1}$.
\end{proof}

Next we prove a splitting property for the functional $I_{\lambda}$, which is related to the Brezis-Lieb type Lemma for nonlocal nonlinearities \cite{AC, MS2}.
\begin{lemma}\label{al2}
Let $c \geq 0$ and $(u_n)$ be a $(PS)_c$ sequence for
	$I_{\lambda}.$ If $u_n \rightharpoonup u$ in $E_{\lambda},$ then
	\begin{eqnarray}
	I_{\lambda}(v_n)-I_{\lambda}(u_n)+I_{\lambda}(u) &=&o_n(1) \label{3e4}\\
	I'_{\lambda}(v_n)-I'_{\lambda}(u_n)+I' _{\lambda}(u) &=&o_n(1),\label{3e5}
	\end{eqnarray}
	where $v_n=u_n-u.$ Furthermore, $(v_n)$ is a
	$(PS)_{c-I_{\lambda}(u)}$ sequence.
\end{lemma}
\begin{proof}
	First of all, note that
\begin{align*}
	I_{\lambda}(v_n)-I_{\lambda}(u_n)+I_{\lambda}(u) &=\frac{1}{2}\left( \|v_n\|^2_{\lambda}-\|u_n\|^2_{\lambda}+\|u\|^2_{\lambda}\right)-&\\
	&-\frac{1}{2p}\int_{\mathbb{R}^3}\left(\int_{\mathbb{R}^3} \frac{\left|v_n(y) \right|^p \left|v_n(x) \right|^p-\left|u_n(y) \right|^p \left|u_n(x) \right|^p+\left|u(y) \right|^p \left|u(x) \right|^p}{\left|x-y \right|^{\mu} }dy\right) dx.&
\end{align*}
Since $u_n \rightharpoonup u$ in  $E_{\lambda}$, we have
\begin{eqnarray}
 I_{\lambda}(v_n)-I_{\lambda}(u_n)+I_{\lambda}(u) =& o_n(1)+\frac{1}{2p}\int_{\mathbb{R}^3}\left( \int_{\mathbb{R}^3} \frac{ \left|v_n(y) \right|^p\left( -\left|v_n(x) \right|^p +\left|u_n(x) \right|^p -\left|u(x) \right|^p \right)}{\left|x-y \right|^{\mu} }dy\right) dx\,\nonumber\\
 &+\frac{1}{2p}\int_{\mathbb{R}^3}\left(\int_{\mathbb{R}^3} \frac{\left|u_n(y) \right|^p\left(  -\left|v_n(x) \right|^p +\left|u_n(x) \right|^p -\left|u(x) \right|^p \right)}{\left|x-y \right|^{\mu} }dy\right)dx \label{3e6}\\
  &+\,\frac{1}{2p}\int_{\mathbb{R}^3}\left(\,\int_{\mathbb{R}^3} \frac{  \left|u(y) \right|^p\left(-\left|v_n(x) \right|^p +\left|u_n(x) \right|^p -\left|u(x) \right|^p \right)}{\left|x-y \right|^{\mu} }dy\,\right)dx\,\nonumber\\
  &+\frac{1}{p}\int_{\mathbb{R}^3}\,\,\,\left(\int_{\mathbb{R}^3}\frac{|v_n(y)|^p|u(x)|^p}{|x-y|^\mu}dy\right)dx.\,\,\,\,\,\,\,\,\,\,\,\,\,\,\,\,\,\,\,\,\,\,\,\,\,\,\,\,\,\,\,\,\,\,\,\,\,\,\,\,\,\,\,\,\,\,\,\,\,\,\,\,\,\,\,\,\,\,\,\,\,\,\,\, \nonumber
 \end{eqnarray}
By the Hardy-Litllewood-Sobolev inequality,
  \begin{eqnarray*}
  \int_{\mathbb{R}^3} \left(\int_{\mathbb{R}^3}\frac{\left|v_n(y) \right|^p\Big( \left|v_n(x) \right|^p -\left|u_n(x) \right|^p +\left|u(x) \right|^p \Big) }{\left|x-y \right|^{\mu}}dy\right)dx
  \leq C \left| v_n\right|^p_{\frac{6p}{6-\mu}} \left( \int_{\mathbb{R}^3}\big| |v_n(x) |^p -|u_n(x) |^p +|u(x) |^p\big|^{\frac{6}{6-\mu}}dx\right)^\frac{6-\mu}{6}.
  \end{eqnarray*}
Notice that
 \begin{eqnarray}
\int_{\mathbb{R}^3}\Big| |v_n(x) |^p -|u_n(x) |^p +|u(x) |^p\Big|^{\frac{6}{6-\mu}}dx=&\int_{B_R(0)}\Big| |v_n(x) |^p -|u_n(x) |^p +|u(x) |^p\Big|^{\frac{6}{6-\mu}}dx+\,\,\,\,\,\,\,\,\,\,\,\,\,\,\,\,\,\,\,\,\,\,\,\,\label{3e7}\\
&\int_{\mathbb{R}^3\setminus B_R(0)}\Big| |v_n(x) |^p -|u_n(x)|^p +|u(x) |^p\Big|^{\frac{6}{6-\mu}}dx.\,\,\,\,\,\,\,\,\,\,\,\,\,\,\,\,\,\,\,\,\,\,\,\,\nonumber
\end{eqnarray}
where, $R>0$  will be  fixed subsequently. As $u_n \rightharpoonup u$ in $E_\lambda$, we know that
 \begin{center}
 	\begin{itemize}
 		\item	$u_n \rightarrow u,\ \mbox{in}\ L^{\frac{6p}{6-\mu}}(B_R(0));$
 		\item $u_n(x) \rightarrow u(x)\ \mbox{a.e. in} \ \mathbb{R}^3,$
 	\end{itemize}
 \end{center}
and  there is $h_1 \in L^{\frac{6p}{6-\mu}}(B_R(0))$ such that
$$
\left|u_n(x)\right|\leq h_1(x) \quad \mbox{a.e. in} \ \mathbb{R}^3.
$$
From this,
$$
\left|v_n(x) \right|^p -\left|u_n(x) \right|^p +\left|u(x) \right|^p  \rightarrow 0\ a.e.\ \mbox{in}\ \mathbb{R}^3,
$$
and
$$\Big| \left|v_n(x) \right|^p -\left|u_n(x) \right|^p +\left|u(x) \right|^p\Big|^{\frac{6}{6-\mu}}\leq (2^p+1)^{\frac{6}{6-\mu}}\left( h_1(x)+\left|u(x) \right|\right)^\frac{6p}{6-\mu}\in L^1(B_R(0)).
$$
Thus, by the Lebesgue Dominated Convergence Theorem,
\begin{equation}\label{3e8}
\int_{B_R(0)}\Big| \left|v_n(x) \right|^p -\left|u_n(x) \right|^p +\left|u(x) \right|^p\Big|^{\frac{6}{6-\mu}}dx \rightarrow 0.
\end{equation}
Furthermore, we also have
$$
\Big|\ \left|u_n(x)-u(x) \right|^p -\left|u_n(x) \right|^p\ \Big| \leq p2^{p-1}(\left|u_n(x) \right|^{p-1}\left|u(x) \right|+\left|u(x) \right|^{p}),
$$
and so,
$$
\int_{\mathbb{R}^3\setminus B_R(0)}\Big| \left|u_n(x)-u(x) \right|^p -\left|u_n(x) \right|^p \Big|^{\frac{6}{6-\mu}}dx \leq C\int_{\mathbb{R}^3\setminus B_R(0)}\left|u_n(x) \right|^{\frac{6(p-1)}{6-\mu}}\left|u(x) \right|^{\frac{6}{6-\mu}}dx+ C\int_{\mathbb{R}^3\setminus B_R(0)}\left| u(x)\right|^{\frac{6p}{6-\mu}}dx.
$$
For $\varepsilon> 0$, we  can choose $R>0$  such that
$$
\int_{\mathbb{R}^3\setminus B_R(0)}\left| u(x)\right|^{\frac{6p}{6-\mu}}dx\leq \varepsilon,
$$
the H\"older inequality combined with the boundedness of $(u_n)$ implies that
\begin{equation}\label{3e9}
\int_{\mathbb{R}^3\setminus B_R(0)}\Big| \left|v_n(x) \right|^p -\left|u_n(x) \right|^p +\left|u(x) \right|^p\Big|^{\frac{6}{6-\mu}}dx\leq \varepsilon.
\end{equation}
Gathering together the boundedness of $(v_n)$ and (\ref{3e7})-(\ref{3e9}), we  deduce that
$$
\int_{\mathbb{R}^3}\Big| |v_n(x) |^p -|u_n(x) |^p +|u(x) |^p\Big|^{\frac{6}{6-\mu}}dx \to 0.
$$
To finish the proof, we need to prove that
$$
\int_{\mathbb{R}^3}\left(\int_{\mathbb{R}^3}\frac{|v_n(y)|^p|u(x)|^p}{|x-y|^\mu}dy\right)dx \to 0.
$$
Once $v_n \rightharpoonup 0$ in $E_{\lambda}$ and  $p \in (2, 6-\mu),$ the sequence $(|v_n|^p)$ is bounded in $ L^\frac{6}{6-\mu}(\R^3)$. As $v_n(x)\to 0$ a.e. in $\R^3$,  we ensure that $|v_n|^p$ converges weakly to $0$ in $ L^\frac{6}{6-\mu}(\R^3)$. Using again the Hardy-Littlewood-Sobolev inequality, we know that the linear functional $F:L^\frac{6}{6-\mu}(\R^3) \to \R$ defined by
$$
F(w)=\int_{\mathbb{R}^3}\big(\frac{1}{|x|^\mu}\ast w\big)|u(x)|^{p}dx
$$
is continuous. Consequently, $F(|v_n|^{p}) \to 0$, or equivalently,
$$
\int_{\mathbb{R}^3}\big(\frac{1}{|x|^\mu}\ast |v_n(x)|^p\big)|u(x)|^{p}dx\to0,
$$
and the proof is complete. \end{proof}

\begin{lemma}\label{al3}
	Let $(u_n)$ be a $(PS)_c$ sequence for $I_{\lambda}$. Then $c=0$, or there exists $c_*>0$, independent of $\lambda,$  such that  $c \geq c_*,$ for all $\lambda >0.$
\end{lemma}

\begin{proof}
	By Lemma $\ref{al1}$, we know  $c \geq 0$. Suppose that $c>0$. On one hand, we know
	\begin{align*}
	c+o_n(1)\|u_n\|_{\lambda}&=I_{\lambda}(u_n)-\frac{1}{2p}I'_{\lambda}(u_n)u_n
	\geq \left( \frac{p-1}{2p}\right)\|u_n\|^{2}_{\lambda},&
	\end{align*}
	equivalently,
	\begin{equation}\label{3e10}
	\limsup_{n \rightarrow + \infty} \|u_n\|_{\lambda}^2 \leq \frac{2pc}{p-1}.
	\end{equation}
	On the other hand, the Hardy-Littlewood-Sobolev inequality together with the Sobolev embedding theorems imply that
	$$
	I'_{\lambda}(u_n)u_n \geq \frac{1}{2}\|u_n\|_{\lambda}^2-K\|u_n\|_{\lambda}^{2p},
	$$
	where $K$ is a positive constant. Thus,  there exists $\delta>0$ such that
	\begin{equation}\label{3e11}
	I'_{\lambda}(u_n)u_n \geq
	\frac{1}{4}\left|\left|u_n\right|\right|_{\lambda}^2,\ \mbox{for}\
	\left|\left|u_n\right|\right|_{\lambda} < \delta.
	\end{equation}
	Consider $c_*= \delta^2\frac{p-1}{2p}$ and
	$c<c_*$. Then it follows that
	\begin{equation}\label{3e12}
	\|u_n\|_{\lambda}\leq \delta
	\end{equation}
	for $n $ large enough. Hence,
	$$
	I'_{\lambda}(u_n)u_n \geq
	\frac{1}{4}\|u_n\|_{\lambda}^2,
	$$
	and thus
	$$
	\|u_n\|_{\lambda}^2 \rightarrow 	0.
	$$
Thereby,
	$$
	I_{\lambda}(u_n) \rightarrow I_{\lambda}(0)=0,
	$$
which contradicts the fact that  $ (u_n) $ is a $ (PS) _c $ sequence with
	$c>0$. Therefore, $c \geq c_*.$
\end{proof}

\begin{lemma}\label{al4}
	Let $(u_n)$ be a $(PS)_c$ sequence for $I_{\lambda}.$ Then,
	there exists $\delta_0 >0$ independent
	of $\lambda,$ such that
	$$
	\liminf_{n
		\rightarrow + \infty}
	\left|u_n\right|_{\frac{6p}{6-\mu}}^{2p} \geq
	\delta_0c.
	$$
\end{lemma}

\begin{proof}
Note that
	\begin{equation*}
	c=\lim_{n \rightarrow +\infty} \left( I_{\lambda}(u_n)-\frac{1}{2}I'_{\lambda}(u_n)u_n\right)= \left( \frac{1}{2}-\frac{1}{2p}\right) \lim_{n \rightarrow +\infty}\int_{\R^3}\Big(\frac{1}{|x|^{\mu}}\ast|u_n|^{p}\Big)|u_n|^{p}dx,
	\end{equation*}
by the Hardy-Littlewood-Sobolev inequality, we obtain
	$$
	c \le \left( \frac{1}{2}-\frac{1}{2p}\right)K\liminf_{n \rightarrow +\infty}\left|u_n \right|_{\frac{6p}{6-\mu}}^{2p}.
	$$
	Therefore,	the conclusion follows by setting
	$$
	\delta_0=\left( \frac{p-1}{2p}\right)K^{-1}>0.
	$$
	\end{proof}

\begin{lemma}\label{al5}
	Let $c_1>0$ be a constant independent of $\lambda$. Given $\varepsilon>0$,
	there exist $\Lambda=\Lambda(\varepsilon)$ and $R=R(\varepsilon,c_1)$ such that,
	if $(u_n)$ is a $(PS)_c$ sequence for $I_\lambda$ with $c \in [0,c_1],$ then
	$$
	\limsup_{n \rightarrow +\infty}\left|u_n\right|_{\frac{6p}{6-\mu},B_R^c(0)}^{2p} \leq \varepsilon,\ \forall \lambda \geq \Lambda.
	$$
\end{lemma}
\begin{proof}
	
	For $R >0$, consider
	$$A(R)=\{x \in \mathbb{R}^3/ \left|x\right|> R\ \mbox{and}\ a(x) \geq M_0\}$$
	and
	$$B(R)=\{x \in \mathbb{R}^3/ \left|x\right|> R\ \mbox{and}\ a(x) < M_0\}.$$
	Then,
	\begin{align}\label{3e13}
		\int_{A(R)}u_n^2dx &\leq \frac{1}{(\lambda M_0 +1)}\int_{\mathbb{R}^3}(\lambda a(x) +1)u_n^2dx&\nonumber\\
		&\leq \frac{1}{(\lambda M_0 +1)}\left|\left|u_n\right|\right|_{\lambda}^2&\\
		&\leq \frac{1}{(\lambda M_0 +1)}\left[ \left( \frac{1}{2}-\frac{1}{2p}\right)^{-1}c+o_n(1)\right]\nonumber&\\
		&\leq \frac{1}{(\lambda M_0 +1)}\left[ \left( \frac{1}{2}-\frac{1}{2p}\right)^{-1}c_1+o_n(1)\right].\nonumber&
	\end{align}
	Once $ c_1 $  is independent of $ \lambda, $  by $(\ref{3e13})$ there is  $\Lambda>0$ such that
	\begin{equation}\label{3e14}
		\limsup_{n\rightarrow +\infty}\int_{A(R)}u_n^2dx <\frac{\varepsilon}{2},\ \forall \lambda \ge \Lambda.
	\end{equation}
	On the other hand, using the H\"older inequality for $s\in \left[1,3\right]$ and the continuous embedding $E_{\lambda}\hookrightarrow L^{2s}(\mathbb{R}^3)$, we see that
	\begin{align*}
		\int_{B(R)}u_n^2dx&\leq \beta \left|\left|u_n\right|\right|_{\lambda}^2\left| B(R)\right|^{\frac{1}{s'}} \leq c_1\left( \frac{1}{2}-\frac{1}{2p}\right)^{-1}\left|
		B(R)\right|^{\frac{1}{s'}}+o_n(1),&
	\end{align*}
	where $\beta$ is a positive constant. Now, by assumption (\ref{a3}) on the potential $a(x)$, we know that
	$$
	\left|B(R)\right|\rightarrow 0,\ \mbox{when}\ R \rightarrow +\infty,
	$$
	then we can choose $R$ large enough such that
	\begin{equation}\label{3e15}
		\limsup_{n \rightarrow +\infty}\int_{B(R)}u_n^2dx< \frac{\varepsilon}{2}.
	\end{equation}
Using (\ref{3e14}) and (\ref{3e15}), we obtain that
	\begin{equation*}
	\limsup_{n \rightarrow +\infty}\int_{\mathbb{R}^3}u_n^2dx< \varepsilon.
	\end{equation*}
	The last inequality combined with interpolation implies that
	\begin{equation*}
	\limsup_{n \rightarrow +\infty}\int_{\mathbb{R}^3\setminus B_R(0)}|u_n|^{\frac{6p}{6-\mu}}dx<\varepsilon,\ \lambda>\Lambda,
	\end{equation*}
	by increasing $R$ and $\Lambda$ if necessary.
\end{proof}
\begin{proposition}\label{p1}
	Given $c_1>0,$ independent of $\lambda$, there exists $\Lambda=\Lambda(c_1)>0$ such that if  $\lambda
	\geq \Lambda$, then $I_{\lambda}$
	verifies the  $(PS)_c$ condition for all $c \in [0,c_1]$.
\end{proposition}

\begin{proof}
	Let $(u_n)$ be a $(PS)_c$ sequence. Lemma \ref{al1} implies that $(u_n)$ is bounded. Passing to a subsequence if necessary,
	$$\left\{ \begin{array}{c}
	u_n\rightharpoonup u,\ \ \mbox{in}\ E_{\lambda};\\
	u_n(x)\rightarrow u(x) , \ \mbox{a.e. in}\ \mathbb{R}^3;\\
	u_n\rightarrow u,\ \mbox{in}\ L^{s}_{loc}(\mathbb{R}^3), 1 \leq s <6.
	\end{array}
	\right.$$
	Then, $I_{\lambda}'(u)=0$ and $I_{\lambda}(u)\geq 0$.
Setting $v_n=u_n-u,$  Lemma \ref{al2} ensures that $(v_n)$ is a $(PS)_{d}$ sequence with $d=c-I_{\lambda}(u)$.
	Furthermore,
	$$0\leq d=c-I_{\lambda}(u)\leq c \leq c_1$$
	We claim that $d=0$. Otherwise, suppose that $d>0$. By Lemma $\ref{al3}$ and Lemma $\ref{al4}$,  we know  $d \geq c_*$
	and
	\begin{equation}\label{3e16}
		\liminf_{n \rightarrow +\infty} \left| v_n\right|_{\frac{6p}{6-\mu}}^{2p} \geq
		\delta_0c_*>0.
	\end{equation}
	Applying Lemma \ref{al5} with $\varepsilon=\frac{\delta_0c_*}{2}>0,$
	there exist $\Lambda,R>0$ such that
	\begin{equation}\label{3e17}
		\limsup_{n \rightarrow +\infty} \left| v_n\right|_{\frac{6p}{6-\mu},B_R^C(0)}^{2p} \leq \frac{\delta_0c_*}{2},\ \mbox{for}\ \lambda \ge \Lambda.
	\end{equation}
	Combining $(\ref{3e16})$ and $(\ref{3e17}),$ we obtain $$\liminf_{n \rightarrow
		+\infty} \left| v_n\right|_{\frac{6p}{6-\mu},B_R(0)}^{2p} \geq
	\frac{\delta_0c_*}{2}>0,$$ which is absurd, because as $v_n
	\rightharpoonup 0$ in $E_{\lambda},$ the compact embedding
	$E_{\lambda}\hookrightarrow L^{\frac{6p}{6-\mu}}(B_R(0))$ implies that
	$$
	\liminf_{n
		\rightarrow +\infty} \left| v_n\right|_{\frac{6p}{6-\mu},B_r(0)}^{2p}=0.
	$$
	Thereby, $d=0$ and $(v_n)$ is a $(PS)_0$ sequence. Hence, by Corollary
	$\ref{ac1}$, $v_n \rightarrow 0$ in $E_\lambda$. Thus, $I_{\lambda}$ 	satisfies the $(PS)_c$ condition for $c \in [0,c_1]$ if $\lambda$ is  large enough.
\end{proof}

\section{The $ (PS)_\infty $ condition}


A sequence $ (u_n) \subset  H^{ 1} ( \mathbb R^3 ) $ is called a $ (PS)_\infty $ \emph{sequence for the family} $ \left( I_\lambda \right)_{\lambda \ge 1} $, if there exist $d \in [0,c_{\Gamma}]$ and a sequence $ ( \lambda_n ) \subset  [1, \infty) $ with $ \lambda_n \to \infty $,  such that
$$
   I_{ \lambda_n }(u_n) \to d \text{ and } \left\| I'_{ \lambda_n }(u_n) \right\|_{E^{*}_{\lambda_n}} \to 0, \text{ as } n \to \infty.
$$

\begin{proposition} \label{(PS) infty condition}
 Suppose that $0<\mu<3$,  $2 \leq  p<6-\mu$ and $ (u_n) \subset H^{ 1 } ( \mathbb R^3 ) $ is a $ (PS)_\infty $ sequence for $ \left( I_\lambda \right)_{\lambda \ge 1} $ with $0<d\leq c_{\Gamma}$. Then, up to subsequence, there exists $ u \in H^{ 1 } ( \mathbb R^3 ) $ such that $ u_n \rightharpoonup u $ in $ H^{ 1}( \mathbb R^3 ) $. Furthermore,
\begin{enumerate}
  \item[(i)] $ u_n \to u $ in $ H^{ 1} ( \mathbb R^3 ) $;
	\item[(ii)] $ u = 0 $ in $ \mathbb R^3 \setminus \Omega $ and $u \in H^{1}_0 (\Omega )$ is a solution for
   $$
		     - \Delta u + u  =  \displaystyle\Big(\int_{\Omega} \frac{ |u|^p}{|x-y|^{\mu}}dy\Big)|u|^{p-2}u\ \   \text{ in } \Omega;
	$$
	 \item[(iii)] $ \displaystyle \lambda_n\int_{\mathbb R^3}  a(x) | u_n |^{ 2 } \to 0 $;
	 \item[(iv)] $ \|u_n-u\|^{2}_{\lambda,\Omega}\to 0$;
   \item[(v)] $  \|u_n\|^{2}_{\lambda,\R^3 \setminus \Omega} \to 0 $;
	 \item[(vi)] $ I_{ \lambda_n } (u_n) \to \displaystyle \frac{1}{2}\int_{ \Omega } (| \nabla u |^{ 2 } +  | u |^{2} )dx- \frac{1}{2p}\int_{ \Omega} \Big(\int_{\Omega} \frac{ |u|^p}{|x-y|^{\mu}}dy\Big)|u|^{p}dx $.
\end{enumerate}
\end{proposition}

\begin{proof}
   By hypothesis,
   $$
   I_{\lambda_n}(u_n)\rightarrow d\ \mbox{and}\ \|I'_{\lambda_n}(u_n)\|_{E'_{\lambda_n}} \rightarrow 0.
   $$
   Then, the same arguments employed in the proof of Lemma \ref{al1}  imply that $( \| u_n \|_{ \lambda_n } ) $ and $ (u_n) $ are  bounded in $ \mathbb R $ and  $ H^{1}( \mathbb R^3 ) $ respectively. And so, up to subsequence, there exists $ u \in H^{1}(\mathbb R^3) $ such that
$$
   u_n \rightharpoonup u  \text{ in } H^{1} ( \mathbb R^3) \, \text{ and } \, u_n(x) \to u(x) \text{ for a.e. } x \in \mathbb R^3.
$$
Now, for each $ m \in \mathbb N $, we define $ C_m = \left\{ x \in \mathbb R^3 \, ; \, a(x) \ge \dfrac{1}{m} \right\} $. Without loss of generality, we may assume that $ \lambda_n < 2 ( \lambda_n-1 ), \, \forall n \in \mathbb N $. Thus
	       $$
		       \int_{ C_m } | u_n |^{ 2 }dx \le \frac{2m}{\lambda_n} \int_{ C_m } \big( \lambda_n a(x)+1) | u_n |^{ 2 }dx \leq \frac{C}{\lambda_n}.
	       $$
By Fatou's lemma, we derive that
	       $$
			     \int_{ C_m } | u |^{ 2 }dx = 0,
	       $$
	       which implies that $ u = 0 $ in $ C_m $, and so, $ u = 0 $ in $ \mathbb R^3 \setminus \overline{\Omega} $. From this, we are able to prove $(i)-(vi)$.
	
$(i)$ \, By a simple computation, we see that
$$
\|u_n-u\|^2_{ \lambda_n}=I'_{\lambda_n}(u_n)u_n-I'_{\lambda_n}(u_n)u+ \int_{\R^3}\big(\frac{1}{|x|^\mu}\ast |u_n|^p\big)|u_n|^{p-2}u_n(u_n-u)dx+o_n(1),
$$
then,
$$
\|u_n-u\|^2_{ \lambda_n}=\int_{\R^3}\big(\frac{1}{|x|^\mu}\ast |u_n|^p\big)|u_n|^{p-2}u_n(u_n-u)dx+o_n(1).
$$
As in the proof of Lemma \ref{al2},
$$
\|u_n-u \|_{\lambda_n}^{2}\rightarrow 0,
$$
which means that $u_n \to u$ in $H^1(\mathbb{R}^3).$

$(ii)$ \, Since $ u \in H^{ 1}(\mathbb R^3) $ and $ u = 0 $  in $ \mathbb R^3 \setminus \overline{\Omega} $, we know $ u \in H^{1}_0( \Omega ) $  and $ u{ |_{\Omega_j} } \in H^{1}_0( \Omega_j) $, for $ j \in \{1,2...,k\}  $. Moreover, taking into account that $u_n \to u$ in $H^{1}(\mathbb R^3)$ and $I'_{\lambda_n}(u_n)\varphi \to 0$  for $\varphi \in C^{\infty}_0 ( \Omega)$, we get
\begin{equation} \label{u is solution}	
	          \int_{\Omega}( \nabla u  \nabla \varphi +  u \varphi )dx - \int_{\Omega} \Big(\int_{\Omega} \frac{ |u|^p}{|x-y|^{\mu}}dy\Big)|u|^{p-2}u \varphi dx = 0,
         \end{equation}
which shows that $ u{ |_{\Omega} }  $ is a solution for the nonlocal problem
         $$
		     - \Delta u + u  =  \displaystyle\Big(\int_{\Omega} \frac{ |u|^p}{|x-y|^{\mu}}dy\Big)|u|^{p-2}u\ \   \text{ in } \Omega.
	$$	
			
$(iii)$ \,  In view of (i),
	       $$
		        \lambda_n \int_{ \mathbb R^3 }  a(x) | u_n |^{2 } dx= \int_{ \mathbb R^3 } \lambda_n a(x) | u_n-u |^{ 2 }dx \leq \|u_n-u\|^{2}_{\lambda_n}.
	       $$
Then
				$$
				\lambda_n \int_{ \mathbb R^3 }  a(x) | u_n |^{2 } dx \to 0.
				$$

$(iv)$ \, For each $ j \in \{1,2...,k\}  $,
	       $$
				    |u_n-u|^{2}_{2, \Omega_j }, |\nabla u_n - \nabla u|^{2}_{2, \Omega_j } \to 0. \quad ( \mbox{see} \quad (i)).
				 $$
				 Therefore,
				 $$
				 \int_{ \Omega } ( | \nabla u_n |^{ 2 } - | \nabla u |^{ 2 } )dx \to 0 \quad \mbox{and} \quad \int_{ \Omega} ( | u_n |^{ 2 } - | u |^{ 2 } )dx \to 0.
				 $$
In view of (iii), we know
				 $$
				    \int_{ \Omega} \lambda_n a(x) | u_n |^{2}  dx \to 0,
			   $$
then
				 $$
				  \|u_n\|^{2}_{ \lambda_n, \Omega } \to \int_{ \Omega } ( | \nabla u |^{ 2 } + | u |^{ 2} )dx.
				 $$

$(v)$ \, Summarizing (i) and $ \|u_n-u\|^{2}_{ \lambda_n } \to 0 $, we obtain
	              $$
								  \|u_n\|^{2}_{ \lambda_n, \mathbb R^3 \setminus \Omega} \to 0.
								$$

$(vi)$ \, We can write the functional  $I_{\lambda_n}$ in the following way
	      $$
\aligned
	          I_{ \lambda_n } (u_n) &=\frac{1}{2} \int_{ \Omega}(| \nabla u_n |^{2} + ( \lambda_n a(x) + 1) | u_n |^{2})dx+ \frac{1}{2} \int_{ \mathbb R^3\setminus \Omega } ( | \nabla u_n |^{ 2 } + ( \lambda_n a(x) +1 ) | u_n |^{ 2 } ) dx\\
&\hspace{1cm}-\frac{1}{2p}\int_{\R^3\backslash  \Omega}\big(\frac{1}{|x|^\mu}\ast |u_n|^p\big)|u_n|^pdx-\frac{1}{2p}\int_{ \Omega}\big(\frac{1}{|x|^\mu}\ast |u_n|^p)|u_n|^pdx.
\endaligned				
$$
				Using $(i)-(v)$, we get
				 $$
					  \frac{1}{2}\int_{ \Omega} ( | \nabla u_n |^{ 2 } + ( \lambda_n a(x) + 1 ) | u_n |^{2 } )dx \to \frac{1}{2}\int_{ \Omega}(| \nabla u|^{2} + | u |^{2})dx,
				 $$

				 $$
					 \frac{1}{2}\int_{ \mathbb R^3 \setminus \Omega} (| \nabla u_n |^{2} + ( \lambda_n a(x) + 1) | u_n |^{2} ) dx\to 0,
				$$
				$$
\int_{\Omega}\big(\frac{1}{|x|^\mu}\ast |u_n|^p\big)|u_n|^pdx\to \int_{\Omega } \Big(\int_{\Omega} \frac{ |u|^p}{|x-y|^{\mu}}dy\Big)|u|^{p}dx,
				$$
				$$
\int_{\R^3\backslash  \Omega}\big(\frac{1}{|x|^\mu}\ast |u_n|^p\big)|u_n|^pdx \to 0.
				$$
				 Therefore, we can conclude that
				$$ I_{ \lambda_n } (u_n) \to \displaystyle \frac{1}{2}\int_{ \Omega } (| \nabla u |^{ 2 } +  | u |^{2} )dx- \frac{1}{2p}\int_{ \Omega } \Big(\int_{\Omega} \frac{ |u|^p}{|x-y|^{\mu}}dy\Big)|u|^{p}dx.$$
\end{proof}

\section{ Further propositions for $c_{\Gamma}$}


In the sequel, without loss of generality, we consider $\Gamma=\{1,\cdots,l\},$ with $l \le k$.  Moreover, let us denote by $\Omega'_\Gamma=\cup_{j \in \Gamma}\Omega'_j$, where $\Omega'_j$ is an open neighborhood of $\Omega_j$ with $\Omega'_j \cap \Omega'_i=\emptyset$ if $j \not=i$. Using this notion, we introduce the functional
$$
   I_{ \lambda,\Gamma}(u) = \frac{1}{2}\int_{ \Omega_{\Gamma}' } (| \nabla u |^{2} +( \lambda a(x) + 1) | u |^{2} )dx-\frac{1}{2p} \int_{ \Omega_{\Gamma}' } \Big(\int_{\Omega_{\Gamma}'} \frac{ |u|^p}{|x-y|^{\mu}}dy\Big)|u|^{p}dx,
$$
which is the energy functional associated to the Choquard equation with Neumann boundary condition
$$
\left\{
\begin{array}{l}
- \Delta u + ( \lambda a(x) + 1) u =\displaystyle \Big(\int_{\Omega'_\Gamma} \frac{ |u|^p}{|x-y|^{\mu}}dy\Big)|u|^{p-2}u, \text{ in } \Omega'_\Gamma,\\
        \frac{\partial u}{\partial \eta} = 0, \text{ on } \partial \Omega'_\Gamma.
\end{array}
\right.
\eqno{(CN_\lambda)}
$$
In what follows, we denote by $c_{\Gamma}$ the number given by
$$
c_{\Gamma}=\inf_{u \in \mathcal{M}_{\Gamma}}I_{\Gamma}(u)
$$
where
$$
\mathcal {M}_{\Gamma}=\{u\in \mathcal N_{\Gamma}: I_{\Gamma}'(u)u_j=0 \mbox { and }
u_{j}\neq 0, \,\,\, \forall j \in \Gamma \}
$$
with $u_{j}=u{|_{ \Omega_j}}$ and
$$
\mathcal{N}_{\Gamma} = \{ u\in
H^1(\Omega_{\Gamma})\setminus\{0\}\, :\,I'_{\Gamma}(u)u=0\}.
$$
Similarly, we denote by $c_{\lambda,\Gamma}$ the number given by
$$
c_{\lambda,\Gamma}=\inf_{u \in \mathcal {M}_{\Gamma}'}I_{ \lambda,\Gamma}(u)
$$
where
$$
\mathcal {M}_{\Gamma}'=\{u\in \mathcal {N}_{\Gamma}': I_{ \lambda,\Gamma}'(u)u_j=0 \mbox { and }
u_{j}\neq 0, \,\,\, \forall j \in \Gamma \}
$$
with $u_{j}=u{|_{ \Omega_j'}}$ and
$$
\mathcal{N}_{\Gamma}' = \{ u\in
H^1(\Omega_{\Gamma}')\setminus\{0\}\, :\,I_{ \lambda,\Gamma}'(u)u=0\}.
$$
Repeating the same arguments in Section 2,  we know that there exist $ w_{\Gamma } \in H_0^{1}( \Omega_{\Gamma} ) $ and $ w_{\lambda,\Gamma } \in H^{1}( \Omega'_{\Gamma} ) $ such that
$$
I_{\Gamma}( w_{\Gamma }) = c_{ \Gamma}	\, \text{ and } \, I'_{\Gamma}( w_{ \Gamma}) = 0
$$
and
$$
I_{ \lambda,\Gamma}( w_{ \lambda,\Gamma }) = c_{ \lambda,\Gamma}	\, \text{ and } \, I'_{ \lambda,\Gamma}( w_{ \lambda,\Gamma}) = 0.
$$

 We have the following proposition, which describes an important relation between $c_{\Gamma}$ and $c_{\lambda,\Gamma}$.
\begin{lemma} \label{CONVER}
   There holds that
\begin{enumerate}
   \item[(i)] $ 0 < c_{\lambda,\Gamma} \le c_{\Gamma} , \, \forall \lambda \geq 0 $;
   \item[(ii)] $ c_{\lambda,\Gamma} \to c_{\Gamma} , \text{ as } \lambda \to \infty$.
\end{enumerate}
\end{lemma}
\begin{proof}
$(i)$ \,  Since $ H^{1}_0(\Omega_{\Gamma}) \subset H^{1}(\Omega'_{\Gamma} )$, it is easy to see that
				 $$
				0 <   c_{\lambda,\Gamma} \leq  c_{\Gamma}.
				 $$

$(ii)$ \,  Let $\lambda_n\to \infty$. From the above commentaries, for each $\lambda_n$ there exists $ w_n \in H^{1}( \Omega') $ with
		     $$
					  I_{ \lambda_n,\Gamma  }(w_n) = c_{\lambda_n,\Gamma}	\, \text{ and } \, I'_{ \lambda_n,\Gamma  } ( w_n) = 0.
				 $$
As $ \big(c_{\lambda_n,\Gamma} \big) $ is bounded, there exists $( w_{ n_i }) $, subsequence of $( w_n) $, such that $ (I_{ \lambda_{ n_i },\Gamma }( w_{ n_i }) )$ converges and  $I'_{ \lambda_{ n_i }, \Gamma} ( w_{ n_i }) = 0 $. Repeating  the same ideas  explored in the proof of  Proposition \ref{(PS) infty condition}, we know that  there exists $ w \in H^{1}_0(\Omega_{\Gamma} ) \setminus \{0\} \subset H^{1}(\Omega_{\Gamma}' ) $ such that
$$
w_j =w|_{\Omega_j} \not=0,  \,\,\, j \in \Gamma
$$	
and 				
				
				$$
				    w_{ n_i } \to w \text{ in } H^{1}( \Omega_{\Gamma}' ), \text{ as }\ n_i \to \infty.
				 $$
				 Furthermore, we also have that
				 $$
				    c_{ \lambda_{ n_i },\Gamma } = I_{ \lambda_{ n_i },\Gamma }( w_{ n_i }) \to I_{\Gamma}(w)
				 $$
				and
				 $$
				    0 = I'_{ \lambda_{ n_i },\Omega' }( w_{ n_i }) \to I_{\Gamma}'(w).
				 $$
By the definition of $c_{\Gamma}$,
				 $$
				    \lim_i c_{ \lambda_{ n_i },\Gamma } \geq c_{\Gamma}.
				 $$
Then, combining the last limit with conclusion (i), we can guarantee that
				 $$
				    c_{ \lambda_{ n_i },\Gamma } \to c_{\Gamma} , \text{ as } n_i \to \infty.
				 $$
				 This establishes the asserted result.
\end{proof}

In the sequel, we denote by $w \in H^{1}_0(\Omega_\Gamma)$ the least energy solution obtained in Section 2, that is,
\begin{equation}\label{GSS}
w \in \mathcal{M}_\Gamma, \quad I_{\Gamma}(w)=c_{\Gamma} \quad \mbox{and} \quad I_{\Gamma}'(w)=0.
\end{equation}
Changing variables by $t_j=s^{\frac{1}{p}}_j$,  it is obvious that
$$
\aligned
I_{\Gamma}\big(t_1w_1+\cdots+t_lw_l\big)&=\sum_{j=1}^{l}\frac{t^2_j}{2}||w_j||^2_j- \frac{1}{2p}\int_{\Omega_{\Gamma}}\left(\displaystyle \int_{\Omega_{\Gamma}}\frac{\big|\sum_{j=1}^{l} t_jw_j\big|^p}{|x-y|^{\mu}}dy\right)\Big| \sum_{j=1}^{l} t_jw_j\Big|^pdx\\
&=\sum_{j=1}^{l}\frac{s^{\frac{2}{p}}_j}{2}||w_j||^2_j- \frac{1}{2p}\int_{\Omega_{\Gamma}}\left( \int_{\Omega_{\Gamma}}\frac{\sum_{j=1}^{l} s_j|w_j|^p}{|x-y|^{\mu}}dy\right)\Big(\sum_{j=1}^{l} s_j|w_j|^p\Big)dx.
\endaligned
$$
Arguing as in \cite{GS},
$$
\int_{\Omega_{\Gamma}}\left(\displaystyle \int_{\Omega_{\Gamma}}\frac{ \sum_{j=1}^{l} s_j|w_j|^p}{|x-y|^{\mu}}dy\right)\Big(\sum_{j=1}^{l} s_j|w_j|^p\Big)dx=\int_{\Omega_{\Gamma}}\left[ \frac{1}{|x|^{\mu/2}}*\Big( \sum_{j=1}^{l} s_j|w_j|^p\Big)\right]^2 dx.
$$
As $s\mapsto s^{2/p}$ is concave and $s\mapsto s^{2}$ is strictly convex, we concluded that the function
$$
G(s_1,s_2,\cdots,s_l)=I_{\Gamma}(s_1^{\frac{1}{p}}w_1+\cdots+s_l^{\frac{1}{p}}w_l)
$$
is strictly concave with $\nabla G(1,\cdots,1)=0$. Hence, $(1,\cdots,1)$ is the unique global maximum point of $G$ on $[0,+\infty)^l$ with $G(1,\cdots,1)=c_{\Gamma}$.
In the sequel, we denote by $w \in H^{1}_0(\Omega_\Gamma)$ the least energy solution obtained in Section 2, that is,
$$
w \in \mathcal{M}_\Gamma, \quad I_{\Gamma}(w)=c_{\Gamma} \quad \mbox{and} \quad I_{\Gamma}'(w)=0.
$$
Assuming $p>2$, there are $r>0$ small enough and $R>0$ large enough such that
\begin{equation} \label{ML1}
	I_{\Gamma}'(\sum_{j=1, j\neq i}^{l}t_j w_j(x)+Rw_i)(Rw_i)<0, \,\,\,\hbox{ for } i \in \Gamma, \forall t_j\in[r,R]\ \  \hbox{and}\ \  j\neq i,
\end{equation}
\begin{equation} \label{ML2}
	I_{\Gamma}'(\sum_{j=1, j\neq i}^{l}t_j w_j(x)+rw_i)(rw_i)>0,\,\,\,\hbox{ for } i \in \Gamma, \forall t_j\in[r,R]\ \  \hbox{and}\ \  j\neq i.
\end{equation}
and
\begin{equation} \label{ML3}
I_\Gamma \big( \sum_{j=1}^{l}t_j w_j(x)\big) < c_{\Gamma}, \ \  \forall (t_1,\cdots, t_l ) \in \partial [r,R]^l, 	
\end{equation}
where $w_j:=w|_{{\Omega}_j }$, $ j \in \Gamma$.
Using these  information, we can define
$$
   \gamma_0 (t_1, \cdots, t_l)(x) = \sum_{j=1}^{l}t_j w_j(x) \in H^{1}_0(\Omega_{\Gamma}), \, \forall ( t_1,\cdots, t_l )\in [r,R]^l.
$$
and denote by $\Gamma_*$ the class of continuous pathes $\gamma \in C\big( [r,R]^l, E_\lambda \setminus \{ 0 \} \big)$ which satisfies the following conditions:
   $$
   \gamma=\gamma_0\ \mbox{on}\ \partial[r,R]^l, \leqno{(a)}
   $$
  and
   $$
   \Phi_{\Gamma}(\gamma)=\frac{1}{2}\int_{\mathbb{R}^3\setminus \Omega_{\Gamma}'}\big(|\nabla \gamma|^2+(\lambda a(x)+1)|\gamma|^2\big)dx-\frac{1}{p}\int_{\mathbb{R}^3\setminus \Omega_{\Gamma}'}\bigg(\frac{1}{|x|^{\mu}}\ast |\gamma|^p\bigg)|\gamma|^pdx\ge 0, \leqno{(b)}
   $$
   where $R>1>r>0$ are the positive constants obtained in \eqref{ML1} and  \eqref{ML2}. Since $\gamma_0 \in \Gamma^*,$ we know  that $\Gamma_* \neq \emptyset$. And by (a) for the path $\gamma$ and \eqref{ML3}, we have
\begin{equation} \label{blambdagamma}
I_\lambda \big( \gamma(t_1,\cdots, t_l) \big) < c_{\Gamma}, \ \  \forall (t_1,\cdots, t_l ) \in \partial [r,R]^l, \forall \gamma \in \Gamma_\ast.
\end{equation}

The following lemma will be used to describe the intersection property of the paths and the set $\mathcal {M}_{\Gamma}$ in the final section.
\begin{lemma} \label{intersection}
   For all $ \gamma \in \Gamma_\ast $, there exists $ (t_1, \ldots, t_l ) \in (r,R)^l $ such that
$$
 I'_{\lambda, \Gamma} ({\gamma} (t_1, \ldots, t_l ) ){\gamma}_j(t_1, \ldots, t_l ) =0,
$$
where ${\gamma}_j(t_1, \ldots, t_l )={\gamma}(t_1, \ldots, t_l ){|_{ \Omega_j'}}$, $ j \in \Gamma$.
\end{lemma}
\begin{proof}
Since $p>2$ and $\gamma = \gamma_0$ on $\partial [r,R]^{l}$, by using of \eqref{ML1} and  \eqref{ML2}, we see that the result follows by  Miranda's Theorem \cite{Miranda}.
\end{proof}

\section{Proof of Theorem \ref{main2}}


In this section, we are ready to find nonnegative solutions $ u_\lambda $ for large values of $ \lambda $, which converges to a least energy solution of $(C)_{\infty,\Gamma}$ as $ \lambda \to \infty $. To this end, we will prove two propositions which, together with  Propositions \ref{(PS) infty condition}, will help us to show the main result in Theorem \ref{main2}.

Henceforth, we denote by
$$
\Theta=\left\{u \in E_\lambda\,:\, \|u\|_{\lambda, \Omega'_j}> \frac{r\tau}{2} \,\,\,  j =1,\cdots, l \right\},
$$
where $r$ was fixed in (\ref{ML1}) and $\tau$ is the positive constant such that
$$
\|u_j\|_j > \tau, \quad \forall u \in \Upsilon_\Gamma=\{u \in \mathcal{M}_\Gamma \,:\, I_\Gamma(u)=c_\Gamma \} \quad \mbox{and} \quad \forall j \in \Gamma.
$$
Furthermore, $I_\lambda^{ c_{\Gamma} }$ denotes the set
$$
	I_\lambda^{ c_{\Gamma} } = \big\{ u \in E_\lambda \, ; \, I_\lambda(u) \le c_{\Gamma} \big\}.
$$
Fixing $\delta=\frac{r\tau}{8}$, for $ \xi >0 $ small enough, we set
\begin{equation} \label{A}
   {\cal A}_\xi^\lambda = \left\{ u \in \Theta_{2\delta} \,:\,\Phi_{\Gamma}(u)\ge 0,\, \left|\left|u\right|\right|_{\lambda,\mathbb{R}^3\setminus \Omega'_{\Gamma}}\leq \xi\, \mbox{and}\, | I_{ \lambda}(u)-c_{\Gamma} | \leq \xi\right\}.
\end{equation}
We observe that
$$
   w \in {\cal A}_\xi^\lambda \cap I_\lambda^{ c_{\Gamma} },
$$
showing that $ {\cal A}_\xi^\lambda \cap I_\lambda^{ c_{\Gamma} } \ne \emptyset $.
We have the following uniform estimate of $ \big\| I'_{ \lambda }(u) \big\|_{E^{*}_{\lambda}} $ on the region $ \left( {\cal A}_{ 2 \xi }^\lambda \setminus {\cal A}_\xi^\lambda \right) \cap I_\lambda^{ c_{\Gamma} } $.

\begin{proposition} \label{derivative estimate}
For each  $\xi> 0 $, there exist $ \Lambda_\ast \ge 1 $ and $ \sigma_0 >0   $ independent of $ \lambda $ such that
\begin{equation}
   \big\| I'_{ \lambda }(u) \big\|_{E^{*}_{\lambda}} \ge \sigma_0, \text{ for } \lambda \ge \Lambda_\ast \text{ and } u \in \left( {\cal A}_{ 2\xi }^\lambda \setminus {\cal A}_\xi^\lambda \right) \cap I_\lambda^{ c_{\Gamma} }.
\end{equation}
\end{proposition}
\begin{proof}
   We assume that there exist $ \lambda_n \to \infty $ and $ u_n \in \left( {\cal A}_{ 2 \xi }^{\lambda_n} \setminus {\cal A}_\xi^{\lambda_n} \right) \cap I_{\lambda_n}^{ c_{\Gamma}  } $ such that
$$
   \big\| I'_{ \lambda_n }(u_n) \big\|_{E^{*}_{\lambda_n}} \to 0.
$$
Since $ u_n \in {\cal A}_{ 2 \xi }^{ \lambda_n } $, we know $( \|u_n\|_{ \lambda_n }) $ and $ \big( I_{ \lambda_n }(u_n) \big) $ are both bounded. Passing to a subsequence if necessary, we may assume that $(I_{ \lambda_n }(u_n)) $ converges. Thus, from Proposition \ref{(PS) infty condition}, there exists $ u \in H^{1}_0( \Omega_{\Gamma}) $ such that $u$ is a solution for
$$
		     - \Delta u + u  =  \displaystyle\Big(\int_{\Omega_{\Gamma}} \frac{ |u|^p}{|x-y|^{\mu}}dy\Big)|u|^{p-2}u\ \   \text{ in } \Omega_{\Gamma}
	$$
with
$$
u_n \to u \,\, \text{in} \,\, H^{1}(\mathbb{R}^{3}), \,\,	 \|u_n\|_{ \lambda_n, \mathbb R^3 \setminus \Omega } \to 0 \, \text{ and } \,  I_{ \lambda_n} (u_n) \to I_{\Gamma}(u).
$$
As $(u_n) \subset \Theta_{2 \delta}$, we derive that
$$
\|u_n\|_{\lambda_n,\Omega'_j} > \frac{r\tau}{4}, \,\,\, j=1,\cdots, l.
$$
Letting  $n \to +\infty$, we get the inequality
$$
\|u\|_j \geq \frac{r\tau}{4}>0, \,\,\, j=1,\cdots, l,
$$
which yields $ u_{ |_{ \Omega_j } } \ne 0 $, $ j=1,\cdots, l$ and $I_{\Gamma}'(u)=0 $.
Consequently, $I_{\Gamma} (u) \geq c_{\Gamma}$. However, from the fact that $ I_{ \lambda_n } (u_n) \leq c_{\Gamma} $ and $I_{ \lambda_n } (u_n) \to I_{\Gamma}(u)$,  we derive that $I_{\Gamma}(u)=c_\Gamma$, and so, $u \in \Upsilon_\Gamma$.  Thus, for $n$ large enough
$$
\|u_n\|_j > \frac{r\tau}{2} \,\,\,  \text{ and } \, \left| I_{ \lambda_n}(u_n)-c_{\Gamma} \right| \leq \xi, \,\,\, j=1,\cdots, l.
$$
So $ u_n \in {\cal A}_\xi^{\lambda_n} $, which is a contradiction, finishing the proof.
\end{proof}

In the sequel, $\xi_1,\xi^*$ will be defined as
$$
\xi_{1}=\min_{(t_1,\cdots, t_l)\in \partial [r, R]^l}|I_{\Gamma}(\gamma_0 (t_1,\cdots, t_l))-c_{\Gamma}|>0
$$
and
$$
\xi^*=\min\{{\xi_1}/{2}, \delta, {\rho}/{2}\},
$$
where $\delta$ was given in(\ref{A}) and
$$
\rho= 4R^2c_{\Gamma},
$$
where $R$ was fixed in (\ref{ML1}). Moreover, for each $s>0$, ${ B}_{s}^\lambda$ denotes the set
$$
{ B}_{s}^\lambda = \big\{ u \in E_\lambda \, ; \, \|u\|^2_{\lambda} \leq s \big\} \,\,\, \text{for} \,\,\, s>0.
$$

\begin{proposition} \label{P}
Suppose that $0<\mu<3$ and  $2 < p<6-\mu$. Let $\xi \in (0, \xi^*)$ and $ \Lambda_\ast \ge 1 $ given in the previous proposition. Then, for $ \lambda \ge \Lambda_\ast $, there exists a solution $ u_\lambda $ of $ (C_\lambda) $ such that $ u_\lambda \in {\cal A}_\xi^\lambda \cap I_\lambda^{ c_{\Gamma} } \cap { B}_{2\rho+1}^\lambda$.

\end{proposition}
\begin{proof}
  Let $ \lambda \ge \Lambda_\ast $. Assume that there are no critical points of $ I_\lambda $ in $ {\cal A}_\xi^\lambda \cap I_\lambda^{ c_{\Gamma} } \cap { B}_{2\rho+1}^\lambda $. Since $ I_\lambda $ verifies the $ (PS)_d $ condition with $0<d\leq c_{\Gamma}$, there exists a constant $ \nu_\lambda > 0 $ such that
$$
   \big\| I'_\lambda(u) \big\|_{E^{*}_{\lambda}} \ge \nu_\lambda, \text{ for all } u \in {\cal A}_\xi^\lambda \cap I_\lambda^{ c_{\Gamma} } \cap { B}_{2\rho+1}^\lambda.
$$
From Proposition \ref{derivative estimate}, we have
$$
   \big\| I'_\lambda(u) \big\|_{E^{*}_{\lambda}} \ge \sigma_0, \text{ for all } u \in \left( {\cal A}_{ 2 \xi }^{\lambda} \setminus {\cal A}_\xi^{\lambda} \right) \cap I_{\lambda}^{ c_{\Gamma} },
$$
where $ \sigma_0 > 0 $ is small enough and it does not depend on $ \lambda $.
 In what follows,  $ \Psi \colon E_\lambda \to \mathbb R $ is a continuous functional verifying
$$
   \Psi(u) =  1, \text{ for } u \in {\cal A}_{\frac{3}{2} \xi}^\lambda \cap \Theta_\delta \cap B^{\lambda}_{2\rho},
$$
$$
\ \Psi(u) = 0, \text{ for } u \notin {\cal A}_{2 \xi}^\lambda \cap \Theta_{2\delta} \cap B^{\lambda}_{2\rho+1}
$$
and
$$
0 \le \Psi(u) \le 1, \, \forall u \in E_\lambda.
$$
We also consider $ H \colon I_\lambda^{ c_{\Gamma} } \to E_\lambda $ given by
$$
   H(u) =
\begin{cases}
   - \Psi(u) \big\| Y(u) \big\|^{ -1 } Y(u), \text{ for } u \in {\cal A}_{2 \xi}^\lambda \cap B^{\lambda}_{2\rho+1}, \\
   \phantom{- \Psi(u) \big\| Y(u) \big\|^{ -1 } Y(u)} 0, \text{ for } u \notin {\cal A}_{2 \xi}^\lambda \cap B^{\lambda}_{2\rho+1}, \\
\end{cases}
$$
where $ Y $ is a pseudo-gradient vector field for $ I_\lambda $ on $ {\cal K} = \left\{ u \in E_\lambda \, ; \, I'_\lambda(u) \ne 0 \right\} $. Observe that  $ H $ is well defined, once $ I'_\lambda(u) \ne 0 $, for $ u \in {\cal A}_{2 \xi}^\lambda \cap I_\lambda^{ c_{\Gamma} } $. The inequality
$$
   \big\| H(u) \big\| \le 1, \, \forall \lambda \ge \Lambda_* \text{ and } u \in I_\lambda^{ c_{\Gamma} },
$$
guarantees that the deformation flow $ \eta \colon [0, \infty) \times I_\lambda^{ c_{\Gamma} } \to I_\lambda^{ c_{\Gamma} } $ defined by
$$
   \frac{d \eta}{dt} = H(\eta), \ \eta(0,u) = u \in I_\lambda^{ c_{\Gamma} }
$$
verifies
\begin{equation} \label{Monotone}
   \frac{d}{dt} I_\lambda \big( \eta(t,u) \big) \le - \Psi \big( \eta(t,u) \big) \big\| I'_\lambda \big( \eta(t,u) \big) \big\| \le 0,
   \end{equation}
\begin{equation} \label{Less}
	 \left\| \frac{d \eta}{dt} \right\|_\lambda = \big\| H(\eta) \big\|_\lambda \le 1
   \end{equation}
and
\begin{equation} \label{eta}
   \eta(t,u) = u \text{ for all } t \ge 0 \text{ and } u \in I_\lambda^{ c_{\Gamma} } \setminus {\cal A}_{2 \xi}^\lambda \cap B^{\lambda}_{2\rho+1}.
\end{equation}
We study now two paths, which are relevant for what follows: \\

$ \noindent \bullet $ The path $ (t_1,\cdots, t_l) \mapsto \eta \big( t, \gamma_0 (t_1,\cdots, t_l) \big), \text{ where } (t_1,\cdots, t_l)\in [r,R]^l $.

\vspace{0.5 cm}

Since $\xi \in (0,\xi^*)$, we have that
$$
   \gamma_0(t_1,\cdots, t_l) \notin {\cal A}_{2 \xi}^\lambda, \, \forall (t_1,\cdots, t_l) \in \partial [r,R]^l.
$$
and
$$
   I_\lambda \big( \gamma_0(t_1,\cdots, t_l) \big) < c_{\Gamma}, \, \forall (t_1,\cdots, t_l) \in \partial [r,R]^l.
$$
Once $\gamma_0(t_1,\cdots, t_l) \in \Theta_{2\delta}$, $\text{for all}\, (t_1,\cdots, t_l) \in [r,R]^l$, \eqref{eta} gives that
$$
 \eta \big( t, \gamma_0(t_1,\cdots, t_l) \big)|_{\Omega'_j}\not=0, \ \ t\geq0.$$
 Moreover, 	it is also easy to see that
 $$
 \frac{1}{2}\int_{\mathbb{R}^3\setminus \Omega'}\big(|\nabla \eta \big( t, \gamma_0 \big)|^2+(\lambda a(x)+1)|\eta \big( t, \gamma_0 \big)|^2\big)dx-\frac{1}{p}\int_{\mathbb{R}^3\setminus \Omega'}\bigg(\frac{1}{|x|^{\mu}}\ast |\eta \big( t, \gamma_0 \big)|^p\bigg)|\eta \big( t, \gamma_0 \big)|^pdx\ge 0.
 $$
Consequently,
$$
\eta \big( t, \gamma_0(t_1,\cdots, t_l) \big) \in \Gamma_\ast,  \ \ t\geq0.
$$

\vspace{0.5 cm}

$ \noindent \bullet $ The path $ (t_1,\cdots, t_l)\mapsto \gamma_0(t_1,\cdots, t_l), \text{ where } (t_1,\cdots, t_l)  \in [r,R]^l $.

\vspace{0.5 cm}

We observe that
$$
   \text{supp} \big( \gamma_0 (t_1,\cdots, t_l) \big)\subset \overline{\Omega_{\Gamma}}
$$
and
$$
   I_\lambda \big( \gamma_0 (t_1,\cdots, t_l) \big) \text{ does not depend on }  \lambda \ge 1,
$$
for all  $(t_1,\cdots, t_l) \in [r,R]^l $. Moreover,
$$
   I_\lambda \big( \gamma_0 (t_1,\cdots, t_l) \big) \le c_{\Gamma}, \, \forall (t_1,\cdots, t_l) \in [r,R]^l
$$
and
$$
   I_\lambda \big( \gamma_0(t_1,\cdots, t_l) \big) = c_{\Gamma} \text{ if, and only if, } t_j = 1, \,  j=1,\cdots,l.
$$
Therefore
$$
   m_0 = \sup \left\{ I_\lambda(u) \, ; \, u \in \gamma_0 \big( [r,R]^l\big) \setminus A_\xi^\lambda \right\}
$$
is independent of $ \lambda $ and $ m_0 < c_{\Gamma} $. In the following, we suppose that there exists $ K_\ast > 0 $ such that
$$
   \big| I_{ \lambda }(u) - I_{ \lambda}(v) \big| \le K_* \| u-v \|_{ \lambda }, \, \forall u,v \in {\cal B}_{2\rho}^\lambda .
$$
Now, we will prove that
\begin{equation} \label{max estimate}
    \max_{(t_1,\cdots, t_l) \in [r,R]^l} I_\lambda \Big( \eta \big( T, \gamma_0 (t_1,\cdots, t_l) \big) \Big) \leq c_{\Gamma}-\frac{\sigma_0 \xi}{2 K_\ast} ,
\end{equation}
for $ T > 0 $ large. In fact, writing $ u = \gamma_0(t_1,\cdots, t_l) $, $ (t_1,\cdots, t_l) \in [r,R]^l $, if $ u \notin A_\xi^\lambda $, from \eqref{Monotone}, we deduce that
$$
   I_\lambda \big( \eta( t, u ) \big) \le I_\lambda (u) \le m_0, \, \forall t \ge 0,
$$
and we have nothing more to do. And so we assume that $ u \in A_\xi^\lambda $ and set
$$
   \widetilde{\eta}(t) = \eta (t,u),\ \  \widetilde{\nu_\lambda} = \min \left\{ \nu_\lambda, \sigma_0 \right\} \ \ \text{ and }\ \  T = \frac{\sigma_0 \xi}{K_\ast \widetilde{\nu_\lambda}}.
$$
Now, we will discuss two cases: \\

\noindent {\bf Case 1:} $ \widetilde{\eta}(t) \in {\cal A}_{\frac{3}{2} \xi}^\lambda \cap \Theta_\delta \cap B^{\lambda}_{2\rho}, \, \forall t \in [0,T] $.

\noindent {\bf Case 2:} $ \widetilde{\eta}(t_0) \notin {\cal A}_{\frac{3}{2} \xi}^\lambda \cap \Theta_\delta \cap B^{\lambda}_{2\rho}, \text{ for some } t_0 \in [0,T] $. \\

\noindent {\bf Analysis of  Case 1}

In this case, we have $ \Psi \big( \widetilde{\eta}(t) \big) = 1 $ and $ \big\| I'_\lambda \big( \widetilde{\eta}(t) \big) \big\| \ge \widetilde{\nu_\lambda} $ for all $ t \in [0,T] $. Hence, from \eqref{Monotone}, we know
$$
   I_\lambda \big( \widetilde{\eta}(T) \big) = I_\lambda (u) + \int_0^T \frac{d}{ds} I_\lambda \big( \widetilde{\eta}(s) \big) \, ds \le c_{\Gamma} -\int_0^T \widetilde{\nu_\lambda} \, ds,
$$
that is,
$$
   I_\lambda \big( \widetilde{\eta}(T) \big) \le c_{\Gamma} - \widetilde{\nu_\lambda}  T \leq c_{\Gamma} -\frac{\sigma_0 \xi}{2K_\ast},
$$
showing (\ref{max estimate}). \\

\noindent {\bf Analysis of Case 2}:  In this case we have the following situations:
\\

\noindent {\bf (a)}: There exists $T_2 \in [0,T]$ such that $\tilde{\eta}(t_2) \notin \Theta_\delta$.  Let $T_1=0$ it follows that
$$
\|\tilde{\eta}(T_2)-\tilde{\eta}(T_1)\| \geq \delta > \mu,
$$
because $\tilde{\eta}(T_1)=u \in \Theta$. \\
\vspace{0.3 cm}
\noindent {\bf (b)}: There exists $T_2 \in [0,T]$ such that $\tilde{\eta}(T_2) \notin B^{\lambda}_{2\rho}$. Let $T_1=0$, we get
$$
\|\tilde{\eta}(T_2)-\tilde{\eta}(T_1)\| \geq \rho > \mu,
$$
since $\tilde{\eta}(T_1)=u \in B^{\lambda}_\rho$. \\

\noindent {\bf (c)}: \, $\tilde{\eta}(t) \in \Theta_\delta \cap B^{\lambda}_{2\rho}$ for all $t \in [0,T]$, and there are $0 \leq T_1 \leq T_2 \leq T$ such that $\tilde{\eta}(t) \in  {\cal A}_{\frac{3}{2} \xi}^\lambda \setminus {\cal A}_\xi^\lambda$ for all $t \in [T_1,T_2]$ with
$$
|I_\lambda(\tilde{\eta}(T_1))-c_{\Gamma}|=\xi \,\,\, \mbox{and} \,\,\, |I_\lambda(\tilde{\eta}(T_2))-c_{\Gamma}|=\frac{3\xi}{2}
$$
From definition of $K_\ast$, we have
$$
   \|\tilde{\eta}(T_2)-\tilde{\eta}(T_1) \| \ge \frac{1}{K_\ast} \big| I_{ \lambda} (\tilde{\eta}(T_2)) - I_{ \lambda} (\tilde{\eta}(T_1)) \big| \ge \frac{1}{2 K_\ast} \xi,
$$
then the mean value theorem implies that $ T_2-T_1 \ge \frac{1}{2 K_\ast} \xi $. Notice that
$$
   I_\lambda \big( \widetilde{\eta}(T) \big) \le I_\lambda(u) -\int_0^T \Psi \big( \widetilde{\eta}(s) \big) \big\| I'_\lambda \big( \widetilde{\eta}(s) \big) \big\| \, ds,
$$
we can deduce that
$$
   I_\lambda \big( \widetilde{\eta}(T) \big) \le c_{\Gamma} -\int_{T_1}^{T_2} \sigma_0 \, ds = c_{\Gamma} - \sigma_0 (T_2-T_1) \le c_{\Gamma} - \frac{\sigma_0 \xi}{2 K_\ast} ,
$$
which proves (\ref{max estimate}).

So, defining $ \widehat{\eta} (t_1,\cdots, t_l) = \eta \big( T, \gamma_0 (t_1,\cdots, t_l) \big) $, $(t_1,\cdots, t_l) \in [r,R]^l $,  we have that $ \widehat{\eta} \in \Gamma_\ast $ and
$$
  \max_{ (t_1,\cdots, t_l) \in [r,R]^l } I_\lambda \big( \widehat{\eta} (t_1,\cdots, t_l) \big)  \leq c_{\Gamma}-\frac{\sigma_0 \xi}{2 K_\ast},  \,\,\, \forall \lambda \geq \Lambda_*.
$$
On the other hand, we can estimate
$$
\aligned
I_\lambda \big( \widehat{\eta} \big)&=\frac{1}{2} \int_{ \mathbb R^3 } ( \left| \nabla  \widehat{\eta}\right|^{ 2 } + ( \lambda a(x) + 1 ) |  \widehat{\eta}|^{ 2 } )dx - \frac{1}{2p}\int_{ \mathbb R^3}\Big(\frac{1}{|x|^\mu}\ast | \widehat{\eta}|^p \Big) | \widehat{\eta}|^pdx\\
&=\frac{1}{2}\int_{ \Omega_{\Gamma}' } (| \nabla \widehat{\eta} |^{2} +( \lambda a(x) + 1) |\widehat{\eta}|^{2} )dx-\frac{1}{2p} \int_{ \Omega_{\Gamma}' } \Big(\int_{\Omega_{\Gamma}'} \frac{ |\widehat{\eta}|^p}{|x-y|^{\mu}}dy\Big)|\widehat{\eta}|^{p}dx\\
&\hspace{5mm}+\frac{1}{2}\int_{ \mathbb R^3 \setminus \Omega_{\Gamma}' } (| \nabla \widehat{\eta} |^{2} +( \lambda a(x) + 1) |\widehat{\eta}|^{2} )dx-\frac{1}{2p}\int_{ \mathbb R^3 \setminus \Omega_{\Gamma}' }\Big(\frac{1}{|x|^\mu}\ast | \widehat{\eta}|^p\Big) | \widehat{\eta}|^pdx\\
&\hspace{1.5cm}-\frac{1}{2p}\int_{\Omega_{\Gamma}' }\Big(\int_{ \mathbb R^3 \setminus \Omega_{\Gamma}' }\frac{| \widehat{\eta}|^p}{|x-y|^\mu}dy \Big) | \widehat{\eta}|^pdx\\
&\geq I_{\lambda, \Gamma}(\widehat{\eta})+\frac{1}{2}\int_{ \mathbb R^3 \setminus \Omega_{\Gamma}' } (| \nabla \widehat{\eta} |^{2} +( \lambda a(x) + 1) |\widehat{\eta}|^{2} )dx-\frac{1}{p}\int_{ \mathbb R^3 \setminus \Omega_{\Gamma}' }\Big(\frac{1}{|x|^\mu}\ast | \widehat{\eta}|^p\Big) | \widehat{\eta}|^pdx
\endaligned
$$
Since $\widehat{\eta} \in \Gamma_*$, it follows that
$$
\aligned
\frac{1}{2}\int_{ \mathbb R^3 \setminus \Omega_{\Gamma}' } (| \nabla \widehat{\eta} |^{2}& +( \lambda a(x) + 1) |\widehat{\eta}|^{2} )dx-\frac{1}{p}\int_{ \mathbb R^3 \setminus \Omega_{\Gamma}' }\Big(\frac{1}{|x|^\mu}\ast | \widehat{\eta}|^p )\Big) | \widehat{\eta}|^pdx\geq0,
\endaligned
$$
and so,
\begin{equation} \label{Low estimate}
    I_\lambda \big( \widehat{\eta} \big)\geq I_{\lambda, \Gamma}(\widehat{\eta}).
\end{equation}
By \eqref{max estimate} and \eqref{Low estimate}, applying Lemma \ref{intersection}, we have
$$
c_{\lambda,  \Gamma}  \leq \max \left\{ m_0, c_{\Gamma} - \frac{\sigma_0\xi}{2 K_\ast}  \right\} \,\,\, \forall \lambda \geq \Lambda_*,
$$
which leads to
$$
\limsup_{\lambda \to +\infty} c_{\lambda,  \Gamma}  \leq  \max \left\{ m_0, c_{\Gamma} - \frac{\sigma_0\xi}{2 K_\ast} \right\} <  c_{\Gamma},
$$
this contradicts with the conclusion $(ii)$ of Lemma \ref{CONVER}.
\end{proof}

\vspace{.5cm}
\noindent {\bf [Proof of Theorem \ref{main2}: Conclusion]}
From the last Proposition, there exists $(u_{\lambda_n})$ with
$\lambda_n \rightarrow +\infty$ satisfying:
\begin{enumerate}
	\item[(a)] $ I'_{ \lambda_n }(u_{\lambda_n}) = 0, \, \forall n \in \Bbb N $;
	\item[(b)] $ I_{ \lambda_n } (u_{\lambda_n}) \to c_{\Gamma}. $
	\item[(c)] $\left|\left|u_{\lambda_n}\right|\right|_{\lambda_n,\mathbb{R}^N\setminus \Omega'_{\Gamma}}\to 0.$
\end{enumerate}	
Therefore, from of Proposition \ref{(PS) infty condition}, we derive that
$( u_{ \lambda_n } )$ converges in $ H^{1}(\Bbb R^3) $ to a function $ u \in H^{1}(\mathbb{R}^3) $, which satisfies $ u = 0 $ outside $ \Omega $ and $ u_{|_{\Omega_j}} \not= 0, \, j =1,\cdots, l$.
Now, we claim that $u=0$ in $\Omega_j$, for all $j \notin \Gamma$. Indeed, it is possible to prove that there is $\sigma_1>0$, which is independent of $j$, such that if $v$ is a  nontrivial solution of $(C)_{\infty,\Gamma}$, then
$$
\|v\|_{H_0^{1}(\Omega_\Gamma)} \geq \sigma_1.
$$
However, the solution $u$ verifies
$$
\|u\|_{H^{1}(\mathbb{R}^{N} \setminus \Omega_\Gamma)}=0,
$$
showing that $u=0$ in $\Omega_j$, for all $j \notin \Gamma$. This finishes the proof of Theorem \ref{main2}.
\qed
\\
\\
\\

\noindent {\bf ACKNOWLEDGMENTS}
The authors would like to thank  the  anonymous referee
for his/her  useful comments and suggestions which help to improve and clarify the paper greatly.

\end{document}